\documentclass[english,12pt]{amsart}

\usepackage{graphicx,epsfig,psfrag}
\usepackage{amssymb,amsfonts,amsmath,amsthm,dsfont,wasysym,pifont,stmaryrd}
\usepackage{epstopdf,yfonts}
\usepackage{psfrag}
\usepackage{tikz}
\input pstricks
\input pst-node

\usepackage{pstricks, pst-node}
\input xy
\usepackage[all]{xy}

\newcommand{\thismonth}{\ifcase\month 
  \or January\or February\or March\or April\or May\or June%
  \or July\or August\or September\or October\or November%
  \or December\fi}

\newcommand{\mcB}{\mathcal{B}}
\newcommand{\mcC}{\mathcal{C}}

\newcommand{\mcH}{\mathcal{H}}

\newcommand{\mcK}{\mathcal{K}}
\newcommand{\mcL}{\mathcal{L}}

\newcommand{\mcS}{\mathcal{S}}
\newcommand{\mcT}{\mathcal{T}}

\newcommand{\mcV}{\mathcal{V}}

\newcommand{\mcX}{\mathcal{X}}

\newcommand{\CC}{\mathbf{C}}
\newcommand{\FF}{\mathbf{F}}

\newcommand{\NN}{\mathbf{N}}
\newcommand{\ZZ}{\mathbf{Z}}
\newcommand{\Ff}{\mathbb{F}}

\newcommand{\tH}{\widetilde{H}}
\newcommand{\tV}{\widetilde{V}}
\newcommand{\tv}{\widetilde{v}}

\newcommand{\VT}{\widetilde{V}}
\newcommand{\BT}{\widetilde{B}}
\newcommand{\HT}{\widetilde{H}}

\DeclareMathOperator{\Aut}{Aut}

\newcommand{\Ga}{\Gamma}

\newcommand{\inv}[1]{{#1}^{-1}}

\newcommand{\Mult}[1]{\mathbf{Mult}(\Gamma_{#1})}

\newcommand{\norm}[1]{\|{#1}\|}

\newcommand{\defn}[1]{\emph{#1}}
\newcommand{\ind}[1]{\operatorname{Ind}_{\Gamma'}^\Gamma {(#1)}}
\newcommand{\st}{:\;} 
\newcommand{\period}{\,.}    

\title[Stability Properties of Multiplicative Representations]{Stability Properties
\\of Multiplicative Representations\\of the Free Group}

\author{Alessandra Iozzi}
\address{Departement Mathematik\\
ETH Z\"urich\\
8092 Z\"urich\\
SWITZERLAND }
\email{iozzi@math.ethz.ch}

\author{M. Gabriella Kuhn}
\address{Dipartimento di Matematica \
Universit\`a di Milano ``Bicocca''\\
Viale Sarca 202\\ 
20126 Milano, ITALIA}
\email{kuhn@matapp.unimib.it}

\author{Tim Steger}
\address{Facolt\`a di Scienze Matematiche Fisiche e Naturali\\
Universit\`a degli Studi di Sassari\\
Via Piandanna 4\\
07100 Sassari, ITALIA}
\email{steger@uniss.it}

\thanks{A.~I. was partial supported by the Swiss National Science Foundation project 2000021-127016/2; 
M.~G.~K. and T.~S. were partially supported by PRIN}
\thanks{{\em Acknowledgements:}  The first named author thanks the Institute Mittag-Leffler in Stockholm
for their warm hospitality in the last phase of the preparation of this paper.  Likewise, the second
named author is grateful to the Forschungsinstitut f\"ur Mathematik at ETH, Z\"urich for its hospitality}

\date{\today}

\newtheorem{theorem}{Theorem}[section]
\newtheorem*{theorem*}{Theorem}
\newtheorem{theorem_intro}{Theorem}
\newtheorem{corollary}[theorem]{Corollary}
\newtheorem{lemma}[theorem]{Lemma}
\newtheorem{proposition}[theorem]{Proposition}

\theoremstyle{definition}
\newtheorem{definition}[theorem]{Definition}
\newtheorem{remark}[theorem]{Remark}

\newtheorem{example}[theorem]{Example}

\numberwithin{equation}{section}

\begin{document}

\parindent=1em\

\pagestyle{myheadings}

\begin{abstract} 
We extend the construction of multiplicative representations for free groups 
introduced in \cite{K-S3}, in such a way that the new class $\mathbf{Mult}(\Gamma)$ of representations so defined is 
stable under taking the finite direct sum, under change of generators (and hence is $\operatorname{Aut}(\Gamma)$-invariant)
under restriction to and induction from a subgroup of finite index.  

The main tool is the detailed study of the properties of the action of a free group on its Cayley graph
with respect to a change of generators, as well as the relative properties 
of the action of a  subgroup of finite index after the choice of a nice 
fundamental domain.

These stability properties of $\mathbf{Mult}(\Gamma)$ are essential in the construction of a new class of
representations for a virtually free group in \cite{Iozzi_Kuhn_Steger_virt}.  
\end{abstract}

\maketitle

\pagestyle{myheadings}

\numberwithin{equation}{section}

\section{Introduction}\label{sec:introduction}
Let $\Gamma$ be a finitely 
generated non-abelian free group. We shall say that a unitary representation
$(\pi,\mcH)$ of a group $G$
 is {\it tempered} if it is weakly contained in the regular 
representation.
In \cite{K-S3}, the second and the third author introduced a new family of
tempered unitary representations of $\Gamma$.
This class is large enough to include all known representations that are obtained by embedding $\Gamma$
into the automorphism group of its Cayley graph.
Beside being rather exhaustive, these representations have interesting properties in their own right,
such as for example beeing representations of the crossed product $C^*$-algebra
$\Ga\ltimes\mcC(\partial\Ga)$ where $\mcC(\partial\Ga)$ is 
the $C^*$-algebra of continuos functions on the boundary $\partial\Ga$ of $\Ga$
(see the discussion after Theorem~\ref{thm_intro:3}).

The definition of these representations requires a set of data, called {\em matrix system with inner product},
consisting of a (complex) vector space and a positive definite sesquilinear form for each generator, 
as well as linear maps between any two pairs of vector spaces, 
all subject to some compatibility condition (recalled in \S~\ref{sec:recall}).

We generalize in this paper the construction in \cite{K-S3} by releasing the condition that the
matrix system with inner product be irreducible (see Definition~\ref{equiv}). 
 The irreducibility 
in \cite{K-S3} insured that, except in sporadic and well understood special 
cases, 
the unitary representations so constructed would be irreducible.  
The starting point in this paper is the following result, 
according to which irreducibility of the matrix system
is not essential: representations arising from
 non-irreducible matrix systems are anyway finitely reducible
in the following sense:

\begin{theorem_intro}
Every representation $(\pi,\mcH)$ constructed from a matrix system with inner products $(V_a,H_{ba},B_a)$
decomposes into the orthogonal direct sum with respect to $\mcB=(B_a)$ 
of a finite number of representations constructed from irreducible matrix systems.
\end{theorem_intro}

We call such a representation {\em multiplicative} and we denote by $\mathbf{Mult}(\Gamma)$ 
the class of representations that are unitarily equivalent to a multiplicative representation
(see the end of \S~\ref{sec:recall} for the precise definition).
That we are allowed to drop the dependence of the set of free generators follows from
the following important result:

\begin{theorem_intro}\label{thm:change} Let $A$ and $A'$ be two symmetric sets of free generators of a free group $\Gamma$,
and let us denote by $\Ff_A$ and $\Ff_{A'}$ the group $\Gamma$ as generated respectively by $A$ and $A'$.
Then for every $\pi\in\mathbf{Mult}{\Ff_{A'}}$ there exists
a matrix system with inner product indexed on $A$, such that $\pi\in\mathbf{Mult}{\Ff_A}$.

In particular the class $\mathbf{Mult}(\Gamma)$ is $\Aut(\Gamma)$-invariant.
\end{theorem_intro}

In \cite{K-S3} the authors give an explicit realization of several known representations, such as
for example the spherical series of Fig\`a-Talamanca and Picardello \cite{FT-P}, 
as multiplicative representations with respect to scalar matrices
acting on one dimentional spaces. At the same time in \cite{Pe-S} 
it is shown that 
if $\pi_s$ and $\Pi_s$ are spherical series representations corresponding to 
different generating sets, say $A'$ and $A$, then they cannot be equivalent
unless $A$ is obtainable by $A'$ by an  automorphism of the Cayley graph
associated to the generating set $A'$.
The above theorem insures that, when we think of a spherical representation 
as a multiplicative representation this pathology disappears, 
in the sense that a spherical representation $\pi_s$ corresponding to a given generating set $A'$ 
will be realized as a multiplicative representation with respect to another generating set $A$
(although in this case the new matrices will fail to be scalars, as on can see in Example~\ref{ex:spherical}).


\medskip
The class $\mathbf{Mult}(\Gamma)$
allows us to define a new class of representations for
 virtually free groups $\Lambda$
(see \cite{Iozzi_Kuhn_Steger_virt}): $\mathbf{Mult}(\Lambda)$ is defined as the class of representations
obtained by inducing to $\Lambda$ a multiplicative representation of a free subgroup of finite index.
The proof that the class $\mathbf{Mult}(\Lambda)$ is independent of the choice of the free subgroup
depends on the following further interesting stability property of the class $\mathbf{Mult}(\Gamma)$.

\begin{theorem_intro}\label{thm_intro:3} Assume that $\Gamma$ is a finitely generated non-abelian free group and 
let $\Gamma'<\Gamma$ be a subgroup of finite index.
\begin{enumerate}
\item If $\pi\in\mathbf{Mult}(\Gamma)$, then the restriction
of $\pi$ to $\Gamma'$ belongs to $\mathbf{Mult}(\Gamma')$.
\item If $\pi\in\mathbf{Mult}(\Gamma')$, then the induction
of $\pi$ to $\Gamma$ belongs to $\mathbf{Mult}(\Gamma)$.
\end{enumerate}
\end{theorem_intro}

Since representations of the class $\mathbf{Mult}(\Gamma)$ are tempered,
the same is true for those of the class
$\mathbf{Mult}(\Lambda)$.

The representations in the class 
$\mathbf{Mult}(\Gamma)$ appear also in a natural way as 
{\em boundary representations}, that is representations
of the cross product $C^\ast$-algebra $\Gamma\ltimes\mcC(\partial\Gamma)$, 
where $\mcC(\partial\Gamma)$ is the $C^\ast$-algebra 
of the continuous functions on the boundary $\partial\Gamma$ of $\Gamma$.
Boundary representations are exactly those which 
admit a {\it boundary realization}, that is,
 a relization as a direct integral over
 $\partial\Ga$ with 
respect to some quasi-invariant measure.

As boundary representations as well, the representations in the class 
$\mathbf{Mult}(\Gamma)$ 
enjoy all of the above properties and this is again an essential ingredient 
in the proof that every representation in the class $\mathbf{Mult}(\Lambda)$
extends to a representation of   $\Lambda\ltimes\mcC(\partial\Ga)$ and hence
admits a boundary realization after identifying
 the two boundaries $\partial\Lambda$ and 
$\partial\Ga$.
Incidentally, it  is proved in  \cite{Iozzi_Kuhn_Steger_virt} that {\it every}
tempered representation of a torsion-free not almost cyclic Gromov hyperbolic group
 has a boundary realization.

However, while the existence of such a boundary realization 
for a representation of a Gromov hyperbolic group follows from 
general $C^\ast$-algebra inclusions as well 
extension properties using Hanh--Banach theorem, and is hence highly
non-constructive, for representations in the class $\mathbf{Mult}(\Gamma)$ 
the boundary realization is more accessible and sometimes
very concrete. Its uniqueness is also studied in details in the scalar case in \cite{K-S2},
but remains in general an open question.

\section{Multiplicative Representations of the Free Group}\label{sec:recall}

Fix a symmetric set $A$ of free generators for~$\Ff_A$, $A=A^{-1}$.
Throughout, when we use $a, b, c,d,a_j$, for $j\in\NN$, for elements
of~$\Ff_A$, it is intended that they are elements of~$A$.  There is a
unique \defn{reduced word} for every $x\in \Ff_A$:
\begin{equation*}
x=a_1a_2\dots a_n \qquad
\text{where for all $j$, $a_j\in A$ and $a_ja_{j+1}\neq e$.}
\end{equation*}
The \defn{Cayley graph} of~$\Ff_A$ has as vertices $\mcV$ the elements
of~$\Ff_A$ and as undirected edges the couples $\{x,xa\}$ for
$x\in \Ff_A$, $a\in A$.  This is a tree $\mcT$ with $\#A$ edges attached to
each vertex and the action of~$\Ff_A$ on itself by left translation
preserves the tree structure.   Since the set of vertices $\mcV$ is independent of the
generating set, whenever we need to emphasize this independence, we identify elements
of the free group with vertices of its associated Cayley graph.

A sequence $(x_0,x_1,\dots,x_n)$ of vertices in the tree is a
\emph{geodesic segment} if for all~$j$, $x_{j+1}$~is adjacent to
$x_j$ and $x_{j+2}\neq x_j$.  We denote such geodesic segment joining
$x_0$ with $x_n$ with
\begin{equation*}
[x_0,x_1,\dots,x_n]\qquad\text{ or }\qquad[x_0,x_n]\,,
\end{equation*}
whenever the intermediate vertices are not important.  
If the vertex $z\in\mcV$ is on the geodesic from $x_0$ to $x_n$,
we write $z\in[x_0,x_n]$.
We define  the distance between two
vertices of the tree as the number of edges in the path joining them.
This gives $d(e,x)=|x|$, $d(x,y)=|x^{-1}y|$.

\begin{definition} A \defn{matrix system} or simply \defn{system}
$(V_a,H_{ba})$ is obtained by choosing
\begin{itemize}
\item a complex finite dimensional vector space~$V_a$ for each $a\in A$, and
\item a linear map $H_{ba}:V_a\to V_b$ for each pair $a,b\in A$, where
$H_{ba}=0$ whenever $ab=e$.
\end{itemize}
\end{definition}
\begin{definition}\label{subsystem} 
A tuple of linear subspaces $W_a\subseteq V_a$ is
called an \emph{invariant subsystem} of  $(V_a,H_{ba})$ if
\begin{equation*}
H_{ba}W_a\subseteq W_b
\qquad\text{ for all $a$, $b$.}
\end{equation*}
For any given invariant subsystem $(W_a,H_{ba})$ 
of  $(V_a,H_{ba})$ 
the \defn{quotient system} $(\widetilde{V}_a,\widetilde{H}_{ba})$
is defined on $\widetilde{V}_a=V_a/W_a$ in the obvious way:
\begin{equation*}\label{quotient}
\widetilde{H}_{ba}\widetilde v_a:=\widetilde{H_{ba}v_a} \qquad\text{where $v_a$ is any representative
for $\widetilde v_a$.}
\end{equation*}

The system $(V_a,H_{ba})$ is called \defn{irreducible} if it is
nonzero and if it admits no invariant subsystems except for itself and 
the zero subsystem.
\end{definition}

\begin{definition}\label{equiv}
A map from the system  $(V_a,H_{ba})$ to the system $(V'_a,H'_{ba})$
is a tuple $(J_a)$ where $J_a: V_a\to V'_a$ is a linear map and
\begin{equation*}
H'_{ab}J_b=J_a H_{ab}\;.
\end{equation*}
The tuple $(J_a)$ is called an \defn{equivalence} if each $J_a$~is
a bijection.  Two systems are called \defn{equivalent} if there is an
equivalence between them.
\end{definition}

\begin{remark}\label{remequiv}
A map $(J_a)$ between irreducible systems
 $(V_a,H_{ba})$ and $(V'_a,H'_{ba})$ is either $0$ or an equivalence.
This  is because the kernels (respectively, the images) of the maps $J_a$ 
constitute an invariant subsystem.
\end{remark}


For $x\in \mcV$  we set once and for all 
\begin{equation}\label{eq:cones}
\begin{aligned}
E(x)&:=\{y\in  \mcV \st \text{the reduced word for }y \text{ ends with }x\}\\
C(x)&:=\{y\in \mcV \st \text{the reduced word for }y \text{ starts with }x\}\\
&\hphantom{:}=\{y\in \mcV\st x\in[e,y]\}\,.
\end{aligned}
\end{equation}

\begin{definition}
A \defn{(vector-valued) multiplicative function} is a function
\begin{equation*}
f:\Ff_A~\to~\coprod _{a\in A}~V_a
\end{equation*}
for which there exists $N=N(f)\geq0$ such that for every $x\in \mcV$, with $|x|\geq N$
\begin{equation}\label{1.1}
\begin{alignedat}{3}
&f(x)\in V_a             &\quad &\text{if } &&x\in E(a) \\
&f(xb)=   H_{ba}f(x)&\quad &\text{if } && x\in E(a) \text{ and}\;|xb| =|x|+1\,.
\end{alignedat}
\end{equation}
\end{definition}

We denote by $\mcH_0^\infty(V_a,H_{ba})$ (or $\mcH_0^\infty$ is there is no risk of confusion)
the space of multiplicative functions with respect to the system $(V_a,H_{ba})$. 

\medskip
Note that if $f$ satisfies \eqref{1.1} 
for some $N=N_0$, it also
satisfies \eqref{1.1} for all $N\geq N_0$.  
We define two multiplicative functions $f$ and $g$ to be equivalent, $f\sim g$,
if $f(x)=g(x)$ for all but finitely many elements of $ \mcV$
and $\mcH^\infty$ is defined as the quotient of the space of multiplicative functions 
with respect to this equivalence relation $\mcH^\infty:=\mcH^\infty_0/\sim$.
The vector space structure on~$\mcH^\infty$ is given by pointwise
multiplication by scalars and pointwise addition, where we choose an arbitrary value
for~$(f_1+f_2)(x)$ for those finitely many~$x$ for which $f_1(x)$ and~$f_2(x)$ do not 
belong to the same space~$V_a$.

\medskip
In the following we will need a particular type of multiplicative function
which we now define.

\begin{definition}\label{shadow}
Let $x$ be a reduced word in $E(a)$ and $v_a\in V_a$. 
A \defn{shadow} $\mu[x,v_a]$ is (the equivalence class of)
a multiplicative function supported on the cone $C(x)$, such that 
\begin{equation*}
N\big(\mu[x,v_a]\big)=|x|\;\text{ and }\mu[x,v_a](x):=v_a\,.
\end{equation*}
\end{definition}

It is clear that every multiplicative function can be written as the sum
of a finite number of shadows. 

\medskip
For each $a\in A$ choose a positive definite sesquilinear
form $B_a$ on $V_a\times V_a$ and set
\begin{equation}\label{eq:norm}
\langle f_1,f_2\rangle:=
\sum_{|x|=N}\;\;\sum_{ \substack{ \;
a\\ |xa|=|x|+1}}
B_a\big(f_1(xa),f_2(xa)\big)
\end{equation}
where $N$ is large enough so that both $f_i$ satisfy \eqref{1.1}.  
It is easy to verify 
that for the definition to be independent of $N$ the $B_a's$ must satisfy the 
condition $B_a(v_a,v_a)=\sum_{b}B_b(H_{ba}v_a,H_{ba}v_a)$,
for all $a\in A$ and $v_a\in V_a$. 

\begin{definition}\label{D-with-inner}
The triple $(V_a,H_{ba},B_a)$ is a \defn{system with inner products}
if $(V_a,H_{ba})$ is a matrix system,  $B_a$~is a positive definite
sesquilinear form on~$V_a$ for each $a\in A$ and for $v_a\in V_a$
\begin{equation}\label{E-cond-B}
B_a(v_a,v_a)=\sum_{b\in A}B_b(H_{ba}v_a,H_{ba}v_a)
  \period
\end{equation}
\end{definition}

We refer to \eqref{E-cond-B} as to a \defn{compatibility condition}.  

\medskip

Assuming that such a family exists define 
a unitary representation $\pi$ of $\Ff_A$  on $\mcH^\infty$ by the rule
\begin{equation}\label{repr}
(\pi(x)f)(y)=f(\inv x y)\period
\end{equation}

The existence of a family of sesquilinear forms satisfying the 
compatibility condition  was shown in \cite{K-S3} as follows.

\begin{definition}
For each $a\in A$, let $S_a$ be the real vector space of symmetric
sesquilinear forms on $V_a\times V_a$.  Let $\mcS=\bigoplus_{a\in A}S_a$.  
We say that a tuple $\mcB=(B_a)\in\mcS$ is
\defn{positive definite} (resp. \defn{positive semi-definite}) if each of its 
components is positive definite (resp. positive semi-definite), 
in which case we write $\mcB>0$ (resp. $\mcB\geq0$).  
\end{definition}

Let $\mcK\subseteq\mcS$ denote the solid cone consisting 
of positive semi-definite tuples.  Define a linear map $\mcL:\mcS\to\mcS$ by the rule
\begin{equation}\label{eq-mcT}
(\mcL \mcB)_a(v_a,v_a)=\sum_{b} B_b( H_{ba}v_a,H_{ba}v_a)\,,
\end{equation}
where $\mcB=(B_a)$, 
and observe that  $\mcL(\mcK)\subseteq\mcK$.

The existence of a tuple~$(B_a)_{a\in A}$
compatible with $(V_a,H_{ba})$ depends on eigenvalues of $\mcL$.
The following lemma summarizes the results  of \cite[\S~4]{K-S3}:

\begin{lemma}[\cite{K-S3}]\label{Vander}
For any given matrix system  $(V_a,H_{ba})$, there exists a positive number 
$\rho$ and a tuple of positive semi-definite sesquilinear forms $(B_a)$
on~$V_a$ such that
\begin{equation*}
\sum_{b}B_b(H_{ba}v_a,H_{ba}v_a)=\rho B_a(v_a,v_a)
  \period
\end{equation*}
If $\lambda$ is any other number such that
$\sum_{b}B_b(H_{ba}v_a,H_{ba}v_a)=\lambda B_a(v_a,v_a)$ then 
$|\lambda|\leq\rho$.

If the matrix system is irreducible then each $B_a$ is strictly
positive definite and, up to multiple scalars, there exists a unique
tuple satisfying \eqref{E-cond-B}.
\end{lemma}

We shall refer to $\rho$ as the \defn{Perron--Frobenius eigenvalue} of the 
system $(V_a,H_{ba})$.

\medskip

As a consequence of the above lemma, it follows that, 
up to a normalization of the matrices $H_{ba}$, 
every matrix system becomes a system with inner products.
Complete now $\mcH^\infty$  to $\mcH=\mcH(V_a,H_{ab},B_a)$  
with respect to the norm defined
in \eqref{eq:norm} 
(where, again, we shall drop the dependence 
from $(V_a,H_{ab},B_a)$ unless necessary) 
and  extend the representation $\pi$ defined in \eqref{repr} to a unitary representation on $\mcH$.

Two equivalent systems $(V_a,H_{ba},B_a)$ and 
$(V'_a,H'_{ba},B'_a)$
give rise to equivalent representations $\pi$ and $\pi'$ on $\mcH=\mcH(V_a,H_{ab},B_a)$
and $\mcH=\mcH(V'_a,H'_{ab},B'_a)$.  In fact, if the tuple $(J_a)$ 
gives  the equivalence of the two systems in 
Definition~\ref{equiv}, the operator
defined by 
\begin{equation*}
U\big(\mu[x,v_a]\big):=\mu[x,J_av_a]
\end{equation*}
for $v_a\in V_a$ extends
to a unitary equivalence
between $(\pi,\mcH(V_a,H_{ab},B_a))$ and $(\pi',\mcH(V'_a,H'_{ab},B'_a))$.
Notice that the converse is not true,
namely non-equivalent systems can give rise to equivalent representations:
the simplest example is given by any spherical representation
of the principal series of Fig\`a-Talamanca and Picardello corresponding to a
non-real parameter $q^{-\frac12+is}$ \cite[Example 6.3]{K-S3}.

\medskip

The irreducibility condition in the last statement in Lemma~\ref{Vander}
is only sufficient.  In fact, even if the matrix system is reducible, 
we can always assume that the $B_a's$ are strictly positive definite
by passing to an appropriate quotient, as the following shows:

\begin{lemma}\label{posdef} Let $(V_a, H_{ba},B_a)$ be a matrix system 
with inner product and let $\pi$ a multiplicative representation on
 $\mcH(V_a,H_{ba},B_a)$. 
Then there exist a matrix system with inner product $(\widetilde V_a,\widetilde H_{ba},\widetilde B_a)$ 
and a representation $\widetilde\pi$ on $\widetilde\mcH(\widetilde V_a,\widetilde H_{ab},\widetilde B_a)$ 
equivalent to $\pi$ such that $\widetilde\mcB=(\widetilde B_a)>0$.
\end{lemma}
\begin{proof}
If $(B_a)$ is not strictly positive definite, then for some $a\in A$, 
\begin{equation*}
W_a:=\{w_a\in V_a\setminus\{0\}\st B_a(w_a,w_a)=0\}\neq\emptyset\,.
\end{equation*}
Since for $w_a\in W_a$
\begin{equation*}
0=B_a(w_a,w_a)=\sum_b B_b(H_{ba}w_a,H_{ba}w_a)
\end{equation*}
and all the terms $B_b(H_{ba}w_a,H_{ba}w_a)$ on the right are
non-negative, each of these must be zero. 
Thus, $H_{ba}w_a\in W_b $ and we conclude that  
$(W_a)$ is a nontrivial invariant subsystem. 

Let $(\widetilde V_a,\widetilde H_{ba})$ be the quotient system.
The tuple $(\widetilde B_a)$ given by
\begin{equation*}
\widetilde B_a(\widetilde v_a,\widetilde v_a)=B_a(v_a,v_a) 
\qquad\text{for some $v_a\in\widetilde v_a$}
\end{equation*}
is well-defined and strictly positive on $(\widetilde V_a)$.
In the representation space $\mcH^\infty(V_a,H_{ba})$
define the invariant subspace
\begin{equation*}
\begin{aligned}
\mcH_W^\infty=\{f\in\mcH^\infty(V_a,H_{ba})\st f(xa)\in W_a\text{ for all }a\in A\text{ and for all }&\\
x\in \Ff _A\text{ with }|x|\geq N(f)\text{ and }|xa|=|x|+1&\}\,.
\end{aligned}
\end{equation*}
and consider the quotient representation $\pi_W$ 
on $\mcH^\infty(V_a,H_{ba})/\mcH^\infty_W$.
Then  the representation space
$\mcH^\infty(V_a,H_{ba})/\mcH^\infty_W$
may be identified with the space $\mcH^\infty(\widetilde V_a,\widetilde H_{ba})$
of vector-valued multiplicative functions
taking values in $\bigoplus_{a\in A} \widetilde V_a$ and, after the appropriate completion, 
$\pi$ is equivalent to $\pi_W$.
\end{proof}

We conclude this section with the definition of the class of representations 
whose stability properties are the subject of study of this paper.

\begin{definition}\label{def:mult}
Given a free group $\Ff_A$ on a symmetric set of generators $A$, 
we say that a representation $(\rho,H)$ belongs to the class {\bf $\mathbf{Mult}{\Ff_A}$} if
there exists a system with inner products $(V_a,H_{ba},B_a)$, a dense 
subspace $M\subseteq H$ and a unitary operator $U:H\to\mcH=\mcH(V_a,H_{ba},B_a)$ such that
\begin{itemize}
\item
$U$ is an isomorphism between $M$ and the space $\mcH^\infty(V_a,H_{ba},B_a)$ 
of vector-valued multiplicative functions.
\item
$U(\rho(x)m)=\pi(x)(Um)$ for every $m\in M$ and $x\in \Ff_A$.
\end{itemize}
\end{definition}

\section{Preliminary Results}\label{sec:preliminaries}
\subsection{The Compatibility Condition and the Norm of a Multiplicative Function}\label{subsec:compatibility}
Let $f$ be a function multiplicative  for $|x|\geq N$.
Fix any vertex $x$ such that $d(e,x)\geq N$ and denote by $t(x)$ the 
last letter in the reduced word for $x$. Then the compatibility condition
can be rewritten as 
\begin{equation}\label{normfv}
B_{t(x)}\big(f(x),f(x)\big)
=\sum_{\substack{\; y\\ |y|=|x|+1}}B_{t(y)}\big(f(y),f(y)\big)\,,
\end{equation}
so that, from \eqref{eq:norm},
\begin{equation*}
\|f\|_{\mcH}^2=\sum_{|x|=N}\|f(x)\|^2\,,
\end{equation*}
where 
\begin{equation*}
\|f(x)\|^2:=B_{t(x)}(f(x),f(x))\,.
\end{equation*}

The hypothesis of compatibility \eqref{E-cond-B} has further consequences in the computation of the norm of a function,
that we illustrate now.  We start with some definitions and notation.

\begin{definition}
Let $\mcT$ be a tree of degree $q+1$ and $\mcX$ a finite subtree. 
We say that $\mcX$ is \defn{non-elementary} if it contains at least two vertices. 
If $x$ is a vertex of $\mcX$, its \defn{degree relative} to $\mcX$ is 
the number of neighborhoods of $x$ that lie in $\mcX$.
A finite subtree $\mcX$ is called \defn{complete} if all its vertices have
relative degree equal either to $1$ or to $q+1$.
The vertices having degree $1$ are called \defn{terminal} while the others are
called \defn{interior}.
\end{definition}

The set of terminal vertices is denoted by $T(\mcX)$.
If $\mcX$ is a complete nonelementary subtree not containing $e$ as an interior
vertex, 
we denote by $\bar x_e$ the unique vertex of $\mcX$ which minimizes the distance
from $e$ and $x_e$ the unique vertex of $\mcX$ connected to $\bar x_e$
(which exists since $\bar x_e\in T(\mcX)$).  
We call $\mcX$ a \defn{complete (nonelementary) subtree based at $x_e$}.
We set moreover $T_e(\mcX):=T(\mcX)\setminus\{\bar x_e\}$ and denote
by $B(x,N) =\{y\in\mcT\st d(x,y)\leq N\}$ the (closed) ball of radius $N$ centered at $x\in\mcT$.

\begin{lemma}\label{formerestrizione}
Let $\mcX$ be any complete nonelementary subtree not containing $e$
as an interior vertex. With the above notation, assume that 
$f$ is a function multiplicative outside the ball $B\big(e,|x_e|\big)$. Then
\begin{equation}\label{formeindotte}
\|f(x_e)\|^2=\sum_{t\in T_e(\mcX)}\|f(t)\|^2\,.
\end{equation}
\end{lemma}

\begin{proof} Let 
\begin{equation*}
n=\sup_{x\in \mcX} d(x_e,x)\period
\end{equation*}
The statement can be easily proved  by induction on $n$. When $n=1$
the subtree $\mcX$ must be exactly $B\big(x_e,1\big)$ and \eqref{formeindotte} 
reduces to \eqref{normfv}.
Assume now that \eqref{formeindotte} is true for $n$ and pick any $y_1$
such that
\begin{equation*}
d(x_e,y_1)=n+1=\sup_{x\in \mcX} d(x_e,x)\period
\end{equation*}
Denote by $[x_e,\dots, \bar y_1,y_1]$ the geodesic joining $x_e$ to $y_1$. 
By construction $y_1$ is a terminal vertex while $\bar y_1$ is an
interior vertex. Let $\mcX_1$ be the subtree obtained from $\mcX$ by removing
all the $q$ neighbors of $\bar y_1$ at distance $n+1$ from $x_e$.
Now $\bar y_1$ is a terminal vertex of $\mcX_1$. 
If the supremum over all the vertices of the new complete subtree $\mcX_1$ 
of the distances $d(x_e,x)$ is  $n$ use induction, otherwise, if
\begin{equation*}
n+1=\sup_{x\in \bar \mcX} d(x_e,x)\;;
\end{equation*}
pick  any $y_2$ such that $n+1=d(x_e,y_2)$ and proceed as before.
In a finite number of steps we shall end with a finite complete subtree
$\mcX_k$ satisfying
\begin{equation*}
n=\sup_{x\in \mcX_k} d(x_e,x)
\end{equation*}
for which \eqref{formeindotte} holds. 
Since by inductive hypothesis $\mcX$ can be obtained from $\mcX_k$ by adding all the $q$ neighbors
of each point $\bar y_i$ which are at  distance $n+1$ from $x_e$, $i=1,\dots,k$, 
again \eqref{formeindotte} follows from \eqref{normfv}.
\end{proof}

We saw that the norm of a multiplicative function can be computed as the sum
of the values of $\|f(x)\|^2$, where $x$ ranges over all terminal vertices in 
$B(e,N)$ for $N$ large enough; building on the previous lemma, 
the next result asserts that branching off in some direction along a complete subtree 
and considering again all terminal vertices does not change the norm.

\begin{lemma}\label{lem:normterm}
Let $\mcX$ be any finite complete subtree containing 
$B(e,N)$ and let $f$ be multiplicative for $|x|\geq N$. Then
\begin{equation*}
\|f\|_{\mcH}^2 = \sum_{x\in T(\mcX)}\|f(x)\|^2\,.
\end{equation*}
\end{lemma}

\begin{proof} Let $L\geq N$ be the radius of the largest ball $B(e,L)$ completely 
contained in $\mcX$, so that $\|f\|_{\mcH}^2=\sum_{|x|=L}\|f(x)\|^2$.

If $B(e,L)\neq \mcX$, the set of points 
\begin{equation*}
I:=\big\{x\in \mcX:\,d(e,x)=L\text{ and }x\notin T(\mcX)\big\}
\end{equation*}
is not empty. Apply now Lemma~\ref{formerestrizione} 
to the complete subtree $\mcX_x$ of $\mcX$ based at $x$ for all $x\in I$.
\end{proof}

\subsection{The Perron--Frobenius Eigenvalue}\label{subsec:perron_frobenius}
Before we conclude this section we prove the following two lemmas, which shed some light 
on the possible values of the Perron--Frobenius eigenvalue of a given matrix system.
Both lemmas, together with Lemma~\ref{posdef}, will be necessary in the proof of 
Theorem~\ref{decompo}.

\begin{lemma}\label{rho<1} Let $(V_a, H_{ba}, B_a)$
 be a matrix system with inner product,
$(W_a, H_{ba})$ an invariant subsystem. Let  
$\pi$ be the multiplicative representation on $\mcH(V_a, H_{ba}, B_a)$
and let $\pi_W$ be the restriction of $\pi$ to a multiplicative representation 
on $\mcH(W_a,H_{ba},B_a)$.  
 Assume that the quotient system
 $(\widetilde V_a,\widetilde H_{ba})$ is irreducible.
If the Perron--Frobenius eigenvalue $\rho$ of the quotient
system $(\widetilde V_a,\widetilde H_{ba})$ 
is less than 1 then the representations $\pi$ and $\pi_W$ are equivalent.
\end{lemma}
\begin{proof}
By Lemma~\ref{posdef} we may assume that the $B_a$'s are 
strictly positive definite.
For each $a$ let
\begin{equation*}
W_a^\perp:=\{v_a\in V_a\st B_a(w_a,v_a)=0\text{ for all }w_a\in W_a\}
\end{equation*}
be the orthogonal complement (with respect to $B_a$) of $W_a$ in $V_a$.
Let $\varphi_a:V_a\to\widetilde V_a$, respectively $P_a:V_a\to W_a^\perp $,
denote  the projection of $V_a$ onto $\widetilde V_a$ and
the orthogonal projection of $V_a$ onto $W_a^\perp $. 
Set $H^\perp_{ba}~:=~P_bH_{ba}P_a$.
The following diagram 
\begin{equation*}
\xymatrix{
 V_a\ar[rr]^{\varphi_a}
& &\widetilde V_a\\
 V_a\ar[u]^=\ar[rr]_{P_a}
& &W_a^\perp\ar[u]_{\varphi_a|_{W_a^\perp}}}
\end{equation*}
is commutative, 
so that the system $(W_a^\perp,H_{ba}^\perp)$
may be viewed as an invariant subsystem of  the quotient system
$(\widetilde V_a,\widetilde H_{ba})$.
Since the dimensions are the same,  the two systems must be equivalent. 

Denote by $\rho$ the Perron-Frobenius eigenvalue
of the system $(\VT_a,\HT_a)$.
By  Lemma~\ref{Vander} there exists an essentially unique
tuple $\widetilde B_a$ of sesquilinear forms on $\widetilde V_a$ such that
\begin{equation}\label{Btilde}
\sum_{b\in A}\BT_b(\HT_{ba}\widetilde v_a,\HT_{ba}\widetilde v_a)=\rho\BT_a(\widetilde v_a,\widetilde v_a)\,,
\end{equation}
which can be chosen to be positive definite 
since the system $(\widetilde V_a,\widetilde B_a)$ is irreducible.
By identifying  the finite dimensional subspaces $ W_a^\perp$ and  $\widetilde V_a$,
the norms induced on $W_a^\perp$ by $B_a$ and on $\widetilde V_a$ by $\widetilde B_a$
are equivalent and there exists a constant $K$ so that
\begin{equation*}
 B_a\big(P_a(v_a),P_a(v_a)\big)
\leq K\BT_a\big(\varphi(v_a),\varphi(v_a)\big)
\end{equation*}
for all $a\in A$.

Define, as in Lemma~\ref{posdef},
\begin{equation*}
\begin{aligned}
\mcH_W^\infty=\{f\in\mcH^\infty(V_a,H_{ba})\st f(xa)\in W_a\text{ for all }a\in A\text{ and for all }&\\
x\in \Ff_A \text{ with }|x|\geq N(f)\text{ and }|xa|=|x|+1&\}\,.
\end{aligned}
\end{equation*}
Under the assumption that $\rho<1$, we shall prove that
$\mcH^\infty_W$
 is dense in $\mcH^\infty(V_a,H_{ba})$
with respect to the norm induced by the $B_a$'s,
from which the assertion will follow.
Choose $f$ in $\mcH^\infty(V_a,H_{ba})$ and $\epsilon>0$. 
Let $N=N(f)$ be such  that $f$ is multiplicative for $n\geq N$ and let
us fix $x\in \Ff_A$ and $a\in A$ such that $|x|\geq N$ and $|xa|=|x|+1$. 
Write $f(xa)=w_a+w_a^\perp$,
where $w_a\in W_a$ and $w^\perp_a\in W^\perp_a$, 
and observe that 
\begin{equation}\label{proj}
\begin{aligned}
P_b\big(f(xab)\big)&=P_b\big(H_{ba}f(xa)\big)=P_b\big(H_{ba}(w_a+w^\perp_a)\big)\\
                   &=P_bH_{ba}w^\perp_a=H^\perp_{ba}w_a^\perp\,.
\end{aligned}
\end{equation}
Define now
\begin{equation*}
g_0:=\sum_{b\st \; ab\neq e}\mu[xab,f(xab)-P_b(f(xab))]
\end{equation*}
and compute
\begin{align*}
\norm{f-g_0}_{\mcH}^2\;
&=\sum_{ \substack{ \;b\\ |xab|=|x|+2}}
  B_b\big(f(xab)-g_0(xab),f(xab)-g_0(xab)\big) \\
&=\sum_{ \substack{ \;b\\ |xab|=|x|+2}} 
  B_b(H^\perp_{ba}w_a^\perp,H^\perp_{ba}w_a^\perp)\\
&\leq K\sum_{\substack{\; b\\ |xab|=|x|+2}}
  \BT_b(H^\perp_{ba}w_a^\perp,H^\perp_{ba}w_a^\perp)\\
&=K\rho\BT_a(w_a^\perp,w_a^\perp)\period
\end{align*}

Let $n$ be large enough so that
\begin{equation*}
K\rho^n\BT_a(w_a^\perp,w_a^\perp)<\epsilon\,.
\end{equation*}

Let $z:=a_1\dots a_n$ a reduced word of length $n$ so that 
$y=xazb$ has length $|y|=|x|+2+n$.
Define
$H^\perp_y=H^\perp_{ba_n}\dots H^\perp_{a_1a}$
and use induction and \eqref{proj} to see that 
\begin{equation*}
P_b(f(y))=H^\perp_yw_a^\perp\,.
\end{equation*}
A repeated application of \eqref{Btilde} yields
\begin{equation*}
\sum_{b\in A}\;\sum_{\substack{y\in{C}(xa)\cap E(b)\\ |y|=|x|+2+n}}
\BT_b(H^\perp_yw_a^\perp,H^\perp_yw_a^\perp)\;=\;
\rho^{n+1}\BT_a(w_a^\perp,w_a^\perp)\,.
\end{equation*}
If we set, as before,
\begin{equation*}
g_n:=\sum_{b\in A}\;\;\sum_{\substack{y\in{C}(xa)\cap E(b)\\ |y|=|x|+2+n}}
\mu[y,f(y)-P_b(y)]\,,
\end{equation*}
then 
\begin{equation*}
\begin{aligned}
 \norm{f-g_n}_{\mcH}^2
&=\sum_{b\in A}\;\;\sum_{\substack{y\in{C}(xa)\cap E(b)\\ |y|=|x|+2+n}}
   B_b\big(P_b(f(y),P_bf(y)\big)\\
&\leq K\sum_{b\in A}\;\sum_{\substack{y\in{C}(xa)\cap E(b)\\ |y|=|x|+2+n}}
\BT_b(H^\perp_yw_a^\perp,H^\perp_yw_a^\perp)\\
&=K\rho^{n+1}\BT_a(w_a^\perp,w_a^\perp)\,,
\end{aligned}
\end{equation*}
and hence
\begin{equation*}
\norm{f-g_n}_{\mcH}^2\leq K\rho^{n+1}\BT_a(w_a^\perp,w_a^\perp)<\epsilon\period
\end{equation*}
Since $g_n$ belongs to $\mcH_W$ this concludes the proof.
\end{proof}

\begin{lemma}\label{rho=1} Let $(V_a,H_{ba},B_a)$ be a matrix system with inner products
and $(W_a,H_{ba})$ a maximal nontrivial invariant subsystem with quotient $(\widetilde V_a,\widetilde H_{ba})$.
Then there exists a tuple of strictly positive definite forms on $\widetilde V_a$
with Perron--Frobenius eigenvalue $\rho=1$.
\end{lemma}

\begin{proof} We may assume that $\mcB:=(B_a)>0$.  
The maximality of $(W_a,H_{ba})$ implies that the quotient system $(\widetilde V_a,\widetilde H_{ba})$
is irreducible, hence by Lemma~\ref{Vander} there exists 
a tuple of strictly positive definite forms $(\widetilde B_a)$ satisfying
\begin{equation*}
\sum_{b}\widetilde B_b(\tH_{ba}\tv_a,\tH_{ba}\tv_a)=\rho \widetilde B_a(\tv_a,\tv_a)
\end{equation*}
for some positive $\rho$.

If the Perron--Frobenius eigenvalue $\rho$ relative to $(\tV_a,\tH_{ba})$  
were strictly smaller than one, 
by Lemma~\ref{rho<1} the representations $\pi$ on $\mcH(V_a,H_{ba},B_a)$
and $\pi_W$ on $\mcH(W_a,H_{ba},B_a)$ would be equivalent and we could restrict ourselves 
to the new system $(W_a,H_{ba},B_a)$ of strictly smaller dimension.  
 
We may assume therefore that $\rho\geq1$.

Assume, by way of contradiction, that $\rho>1$.
Lift the $\widetilde B_a$ to a positive semi-definite form on $V_a$ 
by setting it equal to zero on  $W_a$. Rewrite our conditions
in terms of the operator  $\mcL$ defined in \eqref{eq-mcT}:
\begin{equation*}
\mcL\mcB=\mcB\quad\text{ and} \quad\mcL\widetilde\mcB=\rho\widetilde\mcB
\end{equation*}
where
$\mcB=(B_a)_{a\in A}$ and $\widetilde\mcB=(\widetilde B_a)_{a\in A}$.
Since all the $B_a$ are strictly positive definite, 
there exists a positive number $k$ such that $k B_a-\widetilde B_a$ is strictly  
positive definite on $V_a$ for each $a\in A$. Hence for every integer $n$
\begin{equation*}
\mcL^n(k\mcB-\widetilde\mcB)=k\mcL^n(\mcB)-\mcL^n(\widetilde\mcB)=
k\mcB-\rho^n\widetilde\mcB\geq0
\end{equation*}
Choose  now $v_a\in V_a$ so that $\widetilde B_a(v_a,v_a)\neq0$ and $n$ large enough 
to get a contradiction.
\end{proof}

\section{Stability Under Orthogonal Decomposition }\label{sec:orthogonal_decomposition}
A representation that arises from an irreducible matrix system with inner product
is in most of the cases irreducible or, in some special cases, sum of two irreducible ones.
As mentioned already, this is the situation considered in \cite{K-S3}.
In this section we analyze representations arising from
non-irreducible matrix systems showing that they are still well behaved
as the following theorem shows.

\begin{theorem}\label{decompo}
Every representation $(\pi,\mcH)$ constructed from a matrix system with inner products $(V_a,H_{ba},B_a)$
decomposes into the orthogonal direct sum with respect to $\mcB=(B_a)$ 
of a finite number of representations constructed from irreducible matrix systems.
\end{theorem}
\begin{proof}
Let $(V_a, H_{ba},B_a)$ be a matrix system with inner products
and assume that that $\mcB=(B_a)>0$ (see Lemma~\ref{posdef}).

Let $(W_a,H_{ba})$ be a maximal nontrivial invariant subsystem with
irreducible quotient  $(\tV_a, \tH_{ba})$ and let $(\widetilde B_a)$ be
a tuple of strictly positive definite forms with Perron--Frobenius eigenvalue $\rho=1$,
whose existence follows from Lemma~\ref{rho=1}.
Pull back the forms $(\widetilde B_a)$ to obtain a tuple of positive semi-definite forms
on $V_a$ which have $W_a$ as the kernel and which we still denote by $\widetilde B_a$. 
Define
\begin{equation*}
\lambda_0=\sup\{\lambda>0\st
B_a-\lambda \widetilde B_a\geq0\text{ for all }a\in A\}
\end{equation*}
Since $(B_a)$ are strictly positive $\lambda_0$ is finite.
Moreover, for such $\lambda_0$, $B_a-\lambda_0\widetilde B_a$ 
is not strictly positive for some $a$ and hence, for these $a$'s
\begin{equation*}
W^0_a:= \{ v_a\in V_a\st (B_a-\lambda_0\widetilde B_a) (v_a,v_a)=0\}\neq\{0\}\period
\end{equation*}
Set
\begin{equation*}
(\mcB^0)_a:=B_a-\lambda_0\widetilde B_a
\end{equation*}
and observe that
\begin{equation*}
\begin{aligned}
\mcB^0&=\mcB-\lambda_0\widetilde\mcB\geq0\\
\mcL(\mcB-\lambda_0\widetilde\mcB)&=\mcL\mcB-\lambda_0\mcL\widetilde\mcB=\mcB-\lambda_0\widetilde\mcB\,.
\end{aligned}
\end{equation*}
Arguing as in Lemma~\ref{posdef}  one can  see that also the  $(W^0_a)$,
and hence the ($W_a+W^0_a$), constitute an invariant subsystem.
We claim that $V_a=W_a\oplus W^0_a$.  In fact, 
since $\widetilde B_a|_{W_a}\equiv0$, then $W_a\cap W^0_a=0$ for all $a$. 
Moreover, if $\varphi_a:V_a\to\widetilde V_a$ denotes the projection, 
the system $\varphi_a(W_a\oplus W^0_a)$ would be invariant
and hence, by irreducibility of $(\tV_a)$,
the image $\varphi_a(W_a\oplus W^0_a)$ has to be all of $\tV_a$, 
that is to say $V_a=W_a\oplus W^0_a$ for all $a$. Moreover
\begin{equation*}
B_a=B^0_a+\lambda_0\widetilde B_a
\end{equation*}
is the sum of two orthogonal forms. The representation $(\pi,\mcH)$ constructed
from the system $(V_a,H_{ba},B_a)$ decomposes as the sum of the two
sub-representations corresponding to the systems  $(W_a,H_{ba},B^0_a)$
and $(W^0_a,H_{ba},\widetilde B_a)$ where the latter is an irreducible system.
To complete the proof repeat the above argument for the system $(W_a,H_{ba},B^0_a)$: 
since all the $V_a$ are finite dimensional, this reduction process
will stop with an irreducible subsystem.
\end{proof}

\section{Stability Under Change of Generators}\label{change_of_generators}
Let  $A, A'$ denote  two symmetric set of free generators for the free group and write
$a_i, b_i,c_i$, and $\alpha_j, \beta_j,\gamma_j$, for generic elements of $A$ or $A'$, respectively.
Denote by $\mcT$ and $\mcT'$ the tree relative to the generating set $A$ and $A'$, 
and by $|x|$, $|x|'$ the tree distance of $x$ from $e$ in $\mcT$ and $\mcT'$.

The aim of this section is to prove the following:

\begin{theorem}\label{generators}
Let $\pi\in\mathbf{Mult}(\Ff_{A'})$ be a multiplicative representation
with respect to the set $A'$ of generators.
Then there exists a matrix system with inner product $(V_a,H_{ab},B_a)$ 
indexed on the set of generators $A$, such that $\pi\in\mathbf{Mult}(\Ff_A)$.
\end{theorem}

This allows us to refine the definition of the class of multiplicative representations. 

\begin{definition} Given a non abelian finitely generated free group $\Gamma$, 
we say that a representation $\pi$ belongs to the class $\Mult{}$ if there 
exists a symmetric set of generators $A$ such that $\pi\in\mathbf{Mult}(\Ff_A)$.
\end{definition} 

Observe that the property of being invariant under a change of generators
is enjoyed by the class $\Mult{}$, but not by  single representations, as 
will be shown in the Example~\ref{ex:spherical} at the end of this section.

\medskip
We begin with some definitions.
Every element has a unique expression as a reduced word in both alphabets 
and we shall write $z=a_1\dots a_n$ or
$z=\alpha_1\dots\alpha_k$. 
If $\ell(A,A')$ denotes the maximum length of the elements of $A$ with respect
to the elements of $A'$, then 
\begin{equation*}
|z|'\leq \ell(A,A')|z|\,.
\end{equation*}
We recall from \eqref{eq:cones} that
\begin{equation*}
C(z)=\{y\in\mcV \st z\in [e,y]\} 
\end{equation*}
and we define analogously 
\begin{equation*}
C'(z)=\{y\in\mcV \st z\in[e,y]'\}\,,
\end{equation*}
where $[e,y]'$ denotes the geodesic joining $e$ and $y$ in the tree $\mcT'$.
Hence, if $z=\alpha_1\dots\alpha_k\in\Ff_{A'}$ and $z=a_1\dots a_n\in\Ff_A$,
$C'(z)$ consists of all reduced words in the alphabet $A'$ of the form
$y=\alpha_1\dots \alpha_k s$ with $|y|=k+|s|$ 
while $C(z)$
consists of all reduced words in the alphabet $A$ of the form
$y=a_1\dots a_n s$ with $|y|=n+|s|$.

\medskip
We remark that, for $xy\neq e$,  in general we have that
\begin{equation*}
C(xy)\subseteq xC(y)\,,
\end{equation*}
as $xC(y)$ might contain the identity and hence need not be a cone.
The following lemma gives conditions under which there is, in fact, equality.

\begin{lemma}\label{lem:cones}  Let $x,y\in \mcV$. 
\begin{enumerate}
\item[(i)]$xC(y)=C(xy)$ if and only if $y$ does not belong to the geodesic fom $e$ to $\inv x$
in $\mcT$.
\item[(ii)]\label{prop:3.6}
Let $a\in A$ be such that $|xa|=|x|+1$ and assume that
$C'(y)\subseteq C(a)$. Then $xC'(y)=C'(xy)$.
\end{enumerate}
\end{lemma}
\begin{proof} The identity is not in $xC(y)$ if and only if $x$ does not cancel $y$,
that is, if and only if $y\notin[e,\inv{x}]$.

To prove the second assertion, observe that, since $|xa|=|x|+1$, the element $\inv{x}$ 
does not belong to $C(a)$ and, {\it a fortiori} to $C'(y)$ by hypothesis.
Hence $y$ does not belong the geodesic
$[e,\inv{x}]'$ in $\mcT'$, which, by (i)  is equivalent to saying that $xC'(y)=C'(xy)$.
\end{proof}

The following easy lemma will be useful in the definition of the matrices
and the proof of their compatibility.  

\begin{lemma}\label{lem:last letter} Let $a\in A$ and $z\in \mcV$ such that $C'(z)\subseteq C(a)$.
Then for every $b\in A$, $ab\neq e$, the last letter of $z$ and of $bz$ in the alphabet $A'$ coincide.
\end{lemma}

\begin{proof}
If not, multiplication by $b$ on the left would delete $z$, that is 
the reduced expression in the alphabet $A'$ of the generator $b\in A$ 
would be $b=\alpha_1\dots\alpha_t \inv{z}$. Taking the inverses
one would have $\inv{b}=z\inv{\alpha_t}\dots\inv{\alpha_1}$,
thus contradicting the hypothesis that $C'(z)\subseteq C(a)$.
\end{proof}

We have seen in the last two lemmas 
the first consequences of the inclusion of cones with respect to the two different sets of generators.
Analogous inclusions follow from the fact that, 
given two generating systems $A$ and $A'$, for every 
$k\geq0$ there exists an integer $N=N(k)$ 
such that the first $N(k)$ letters of a word $z$ in the alphabet $A'$
determine the first $k$ letters of $z$ in the alphabet $A$. In other words,
for any given $z\in\mcV$ there exists $N(|z|)$ and $y$ with $|y|'\leq N(|z|)$
so that
\begin{equation}\label{eq:incl}
C'(y)\subseteq C(z)\,.
\end{equation}
The set of $y\in\mcV$ with this property is not necessarily unique.
To refine the study of the consequences of this cone inclusion, 
we need to consider, among the $y$ that satisfy \eqref{eq:incl},
those that are the ``shortest'' with this property, in the appropriate sense.
To make this precise, we use the following notation:
\begin{equation*}
\begin{aligned}
\bar{y}&\text{ is the last vertex before }y\text{ in the geodesic }[e,\dots,\bar{y},y]'\subset\mcT'\\
\widetilde y_z&\text{ is the first vertex in the geodesic }[e,y]'\text{ such that }
C'(\widetilde y_z)\subseteq C(z)\,.
\end{aligned}
\end{equation*}
(For ease of notation, we will remove the subscript $z$ whenever this does not cause
any confusion.)
For any $z\in\mcV$ we then define
\begin{equation*}
\begin{aligned}
Y(z)=&\left\{y\in \mcV \st C'(y)\subseteq C(z)\text{ and }
C'(\bar y)\nsubseteq C(z)\right\}\\
=&\left\{y\in \mcV\st C'(y)\subseteq C(z)\text{ and }
y=\widetilde{y_z}\right\}
\end{aligned}
\end{equation*}

Then we have the following analogue of Lemma~\ref{lem:cones}:

\begin{corollary}\label{corYab}
For every $a,b\in A$, $ab\neq e$, we have 
\begin{equation*}
aY(b)=Y(ab)\,.
\end{equation*}
\end{corollary}
\begin{proof}
Let $y\in Y(b)$. By Lemma~\ref{lem:cones}(ii) , $\overline{ay}=a\bar y$. 
Since $C'(y)\subseteq C'(\bar y)\nsubseteq C(b)$
and  $C'(\bar y)\supseteq C'(y)$ there exists a reduced word
$\bar y t$ in the alphabet $A'$ so that $\bar y t\in C(d)$ for some
$d\in A$ with $d\neq b$. Hence the element $a\bar y t$
will not be contained in $C(ab)$.
\end{proof}

For any given $\pi'$ in $\mathbf{Mult}(\FF_{A'})$ we shall now  construct $\pi$ in 
$\mathbf{Mult}(\Ff_A)$ so that $\pi'$ is either a subrepresentation or a quotient of $\pi$.
Namely, if we are given a matrix system with inner products 
$(V'_\alpha,H'_{\beta\alpha},B'_\alpha)$, we need to define a new system
$(V_a, H_{ba}, B_a)$ in such a way that the original system appears as a 
quotient or as a subsystem of the new one.

\begin{definition}\label{serveperexss}
Let $z=\alpha_1\dots\alpha_{k-1}\alpha_k\in\Ff_{A'}$ and define
\begin{equation*}
V'_z=V'_{\alpha_k}\quad   B'_z=B'_{\alpha_k}\,.
\end{equation*}
We set
\begin{equation*}
V_a=\bigoplus_{z\in Y(a)} V'_z \qquad  B_a=\bigoplus_{z\in Y(a)} B'_z 
\end{equation*}
\end{definition}

We need now to define the new matrices $H_{ba}:V_a\to V_b$, for $b\neq\inv a$.
To this extent, take $z\in Y(b)$. Since $b\neq\inv a$, then $az\in C(a)$
and hence, by definition, $\widetilde{(az)}_a\in Y(a)$.
Then we have two cases: either $az=\widetilde{(az)}_a$ and hence $az\in Y(a)$;
or $az=\widetilde{(az)}_a\,x$ with $x\neq e$.  In this case, if the reduced expression
for $x$ in the alphabet $A'$ is $x=\alpha_1\dots\alpha_n$ and $\alpha\neq\inv\alpha_1$
is the last letter (in $A'$) of $\widetilde{(az)}_a\,$, define
\begin{equation*}
H'_{az,\widetilde{az}}
:=H'_{\alpha_n\alpha_{n-1}}\dots H'_{\alpha_1\alpha}
\end{equation*}
where we wrote ${az,\widetilde{az}}$ for ${az,\widetilde{(az)}_a}$ for 
ease of notation.
The new matrices $H_{ba}:V_a\to V_b$ can hence be defined to be block matrices
indexed by pairs $(z,w)$, with $z\in Y(b)$ and $w\in Y(a)$, as follows:
\begin{equation}\label{matricicambio}
(H_{ba})_{z,w}:=
\begin{cases}
\operatorname{Id}&\text{ if }w=az=\widetilde{(az)_a}\\
H'_{az,\widetilde{az}}&\text{ if }w=\widetilde{(az)_a}\neq az
\end{cases}
\end{equation}
and $(H_{ba})_{z,w}=0$ for all other $w\in Y(a)$ with $w\neq \widetilde{(az)_a}$.

\medskip
In the course of the definition we have shown that 
\begin{equation*}
\bigcup_{\substack {z\in Y(b)\\ b\neq\inv a}}\widetilde{(az)}_a\subseteq Y(a)\,,
\end{equation*}
but to show that the matrices so defined give a compatible matrix system
we need to show that the above inclusion is in fact an equality,
namely:

\begin{proposition}\label{prop:3.10}
We have that 
\begin{equation*}
Y(a)=\bigcup_{\substack {z\in Y(b)\\ b\neq\inv a}}\widetilde{(az)}_a
\end{equation*}
\end{proposition}

\begin{proof}
Take any $w\in Y(a)$ so that $C'(w)\subseteq C(a)$. 
Hence either there exists $b\neq\inv a$ such that $C'(w)~\subseteq~C(ab)$, 
in which case $w\in Y(ab)$, 
or $C'(w)~\nsubseteq~C(ab)$ for all $b\neq \inv a$. In this case, 
according to the discussion after Lemma~\ref{lem:last letter}, 
there exists $b\neq\inv a$ and $t_b\in \mcV$ with the 
following properties:
\begin{enumerate}
\item $|wt_b|'=|w|'+|t_b|'$;
\item $C'(wt_b)\subseteq C(ab)$;
\item $t_b$ is minimal with the above properties, that is $C'(w\overline t_b)\nsubseteq C(ab)$. 
\end{enumerate}
In the last case one has, by definition, $wt_b\in Y(ab)$.
By Corollary~\ref{corYab} $Y(ab)=aY(b)$, so that either
$w=az$ or $wt_b=az$  for some $z\in Y(b)$. Since $w\in Y(a)$, it is obvious that
$w=\widetilde{(az)}_a$ when $w=az$. To finish we must
show that $w=\widetilde{(wt_b)}_a$ when $wt_b=az$. 
By definition $\widetilde{(az)}_a$
is the first vertex in the geodesic $[e,wt_b]'=[e,az]'$ such that $C'(\widetilde{(az})_a)\subset C(a)$.
But by hypothesis $w\in Y(a)$, that is $C'(w)\subset C(a)$ and $C'(\overline w)\nsubseteq C(a)$.
Thus $\widetilde{(az)}_a=w$.
\end{proof}

\hskip-2cm
	{ \begin{tikzpicture}[scale=0.8]
  \draw[very thick] 
  	(-1.14,-1.63)  -- (0,0)  -- (10,0)   
  	(2,0) -- ++(75:5) 
	(2,0) -- ++(285:5) 
	(4,0) -- ++(40:5)  
	(4,0) -- ++(320:5)
  	(2.5,-1.7) -- ++(335:4.5) 
	(6,0) -- ++(45:2) -- (7.5,1.5) -- (10.5,1.5);
  \draw  (-.2,.2) node {\tiny{$e$}} 
  			(1.8,.3) node {\tiny{$a^{-1}$}} 
			(3.5,.3) node {\tiny{$a^{-2}=\overline w$}}
			(7.1,-.3) node {\tiny{$a^{-3}=w\in Y_0(a^{-2})$}}
			(10,.3) node {\tiny{$a^{-4}=w\alpha=wt_{a^{-1}}\in Y_0(a^{-2})$}}
			(2.9,-1.414) node {\tiny{$a^{-1}b=w\beta=wt_b\in Y_0(a^{-1}b)$}}
			(-1.8,-1.8) node {\tiny{$\overline w\beta=b$}}
  			(6,-3.9) node {\tiny{${C(a^{-1}b)}$}} 
			(4.5,2) node {\tiny{$C(a^{-1})$}}
			(8.3,2.4) node {\tiny{$C(a^{-2})$}}
			(7,1.8) node {\tiny{$a^{-3}b^{-1}$}}
			(8.8,1.8) node {\tiny{$a^{-3}b^{-1}a^{-1}$}}
			(13,1.8) node {\tiny{$a^{-3}b^{-1}a^{-2}=w\beta=wt_{a^{-1}}\in Y_0(a^{-2})$}};
  \draw[thick, red] 
  	(-1.14,-1.63) -- (4,0) -- (10,0) 
	(2.5,-1.7) -- (6,0) -- (10.5,1.5) 
	(2.5,-1.7) -- ++(300:3.5) 
	(2.5,-1.7) -- ++(320:4);
  \draw[red] (4.8,-4.3) node {\tiny{${C'(a^{-1}b)}$}};
  \filldraw[ultra thick] 
  	(0,0) circle [radius=.05] 
	(-1.14,-1.63) circle [radius=.05] 
	(2,0) circle [radius=.05] 
	(4,0) circle [radius=.05] 
	(6,0) circle [radius=.05] 
	(8,0) circle [radius=.05] 
	(2.5,-1.7) circle [radius=.05] 
	(7.5,1.5) circle [radius=.05] 
	(9,1.5) circle [radius=.05] 
	(10.5,1.5) circle [radius=.05] ;
    \end{tikzpicture}}
\begin{center}{\sc Figure 1}: The trees $\mcT$  (in black) and $\mcT'$ (in red) associated respectively to $\Ff_A$ and $\Ff_{A'}$, where
$A=\{a,b,a^{-1}, b^{-1}\}$ and $A'$ is obtained with the change of generators $a\mapsto\alpha$ and $b\mapsto\beta=a^2b$. \end{center}

\vskip1cm
In the course of the proof of the above proposition we have distinguished two types
of elements of $Y(a)$, and we can consequently conclude the following:

\begin{corollary}\label{cor:y0y1}We have
\begin{equation*}
Y(a)=Y_0(a)\sqcup Y_1(a)\,,
\end{equation*}
where
\begin{equation*}
\begin{aligned}
Y_1(a):
&= \bigcup_{b\neq\inv a} \left(Y(a)\cap Y(ab)\right)\\
&=\big\{w\in Y(a)\st \text{there exists }b\neq \inv a\text{ and }z\in Y(b),\text{ such that}\\ 
&\hphantom{hhhhhhhhhhhhhhhhhhhhhhhhhhhhhhhhhhhh}
w={az}=   \widetilde{(az)}_a\big\}\\
\end{aligned}
\end{equation*}
and
\begin{equation*}
\begin{aligned}
Y_0(a):
 &=\big\{w\in Y(a)\st \text{for all }b\neq\inv a,\,C'(w)\nsubseteq C(ab)\big\}\\
&=\big\{w\in Y(a)\st\text{for some }b\neq\inv a\text{ there exists }\,z\in Y(b), \text{ such}\\
&\hphantom{hhhhhhhhhhhhhhhh}\text{that }w=\widetilde{(az)}_a\;
\text{ and }az=w\,x,\text{ with }x\neq e\big\}\,.
\end{aligned}
\end{equation*}
\end{corollary}
\smallskip

To prove the compatibility condition we will make use of Lemma~\ref{formerestrizione},
so that we need to construct an appropriate finite complete subtree in $\mcT'$.
Notice that for all $w\in \mcV$, the set $\overline w\cup C'(w)$ is a complete subtree, 
but infinite.  To "prune" it so that it will be finite and still complete, consider
an element $w\in Y_0(a)$ and  the following decomposition

\begin{equation*}
\begin{aligned}
   C'(w)
=&\big\{y\in C'(w)\st\, C'(y)\nsubseteq C(ab)\text{ for all } b\neq \inv{a}\big\}\\
&\cup \big\{y\in C'(w)\st\, C'(y)\subseteq C(ab)\text{ for some }b\neq \inv{a}\big\}\\
=&I'_w\cup\bigcup_{b\neq \inv a} \big\{y\in C'(w)\st C'(y)\subseteq C(ab)\big\}\,,
\end{aligned}
\end{equation*}
where we have set
\begin{equation*}
I'_w:=\big\{y\in C'(w)\st\, C'(y)\nsubseteq C(ab)\text{ for all }b\neq\inv a\big\}\,.
\end{equation*}

Since the set $I'_w$ is finite and $w\in I'_w$,  we need to prune the other set.

\begin{proposition}\label{Z'}
Let $w\in Y_0(a)$ and define
\begin{equation*}
\begin{aligned}
T'_w:=&\bigcup_{b\neq \inv a} \big\{y\in C'(w)\st C'(y)\subseteq C(ab),\,C'(\overline y)\nsubseteq C(ab)\big\}\\
=&\bigcup_{b\neq \inv a}\big(C'(w)\cap Y(ab)\big)\,.
\end{aligned}
\end{equation*}
The set
\begin{equation*}
\mathcal{X}'_w:=\{\overline w\}\cup I'_w\cup T'_w
\end{equation*}
is a finite complete subtree in $\mcT'$ whose terminal vertices are $\overline w$ and $T'_w$.
\end{proposition}

Before proceeding to the proof, we remark that this kind of construction will be performed
also in other parts of the paper, whenever we need to construct a finite complete subtree
(see for example Lemmas~\ref{lem:subtree}, \ref{lem:tree-subtree} and \ref{lem:done} in \S~\ref{subsec:induction}).

\begin{proof}
By definition if $y\in I'_w\setminus\{w\}$, then $\overline y\in I'_w$ and if $y\in T'_w$, then
$\overline y\in I'_w$.  This shows in particular that $T'_w\subset T(\mcX'_w)$.  To see that 
the set of terminal vertices consists of $\{\overline w\}\cup T'_w$, observe that if 
$y\in I'_w$ and $y\alpha\in\mcT'$ is such that $|y\alpha|'=|y|'+1$, then by construction
either $y\alpha\in I'_w$ or $y\alpha\in T'_w$. 
\end{proof}

We are now finally ready to prove the compatibility condition.

\begin{proposition}
The system $(V_a,B_a, H_{ba})$ is a compatible matrix system in
the sense of \eqref{E-cond-B}.
\end{proposition}
\begin{proof}
We need to show that if $v_a\in V_a$, then
\begin{equation}\label{compcambio}
B_a(v_a,v_a)=\sum_{b\neq\inv a}B_b(H_{ba}v_a, H_{ba}v_a)\,.
\end{equation}

As in \eqref{prop:3.10} write 
\begin{equation*}
Y(a)=\bigcup_{\substack {z\in Y(b)\\ b\neq\inv a}}\widetilde{(az)}_a=
Y_0(a)\bigcup Y_1(a)\period
\end{equation*}

By definition of $B_a$ and by Corollary~\ref{cor:y0y1} we can write the left hand side as
\begin{equation*}
B_a(v_a,v_a)=\sum_{w\in Y_0(a)}B'_w(v'_w,v'_w)+\sum_{w\in Y_1(a)}B'_w(v'_w,v'_w)
\end{equation*}
and, likewise the right hand side as
\begin{equation*}
\begin{aligned}
&  \sum_{b\neq\inv a}B_b(H_{ba}v_a, H_{ba}v_a)
=\sum_{b\neq\inv a}\sum_{z\in Y(b)}\sum_{\substack{w=\widetilde{az}}}
B'_z(H'_{az,\widetilde{az}}\,v'_w,H'_{az,\widetilde{az}}\,v'_w)=\\
&\sum_{b\neq\inv a}
\sum_{z\in Y(b)}\sum_{\substack{w=\widetilde{az}\neq az\\w\in Y_0(a)}}
B'_z(H'_{az,\widetilde{az}}\,v'_w,H'_{az,\widetilde{az}}\,v'_w)\\
&+\sum_{b\neq\inv a}\sum_{z\in Y(b)}
\sum_{\substack{w=\widetilde{az}=az\\w\in Y_1(a)}}
 B'_z(v'_w,v'_w)       \;,
\end{aligned}
\end{equation*}
where we used the  definition of the $H_{ba}$ \eqref{matricicambio}.

Write $Y_1(a)=\coprod_{b\st b\neq\inv a}(Y(a)\cap Y(ab))
$, a disjoint
union.  Since, for every $b\neq\inv a$, the set $Y(a)\cap Y(ab)$ consists
of those elements $w$ of the form $w=az=\widetilde{az}$ for some $z\in Y(b)$,
using Lemma~\ref{lem:last letter} we get
\begin{equation*}
  \sum_{w\in Y_1(a)}B'_w(v'_w,v'_w)
=\sum_{b\neq\inv a}\sum_{z\in Y(b)}\sum_{w=az\in Y_1(a)} B'_{az}(v'_w,v'_w)\,,
\end{equation*}

so that showing \eqref{compcambio} reduces to showing that
\begin{equation*}
\sum_{w\in Y_0(a)}B'_w(v'_w,v'_w)=\sum_{b\neq\inv a}\sum_{z\in Y(b)}
\sum_{w=\widetilde{az}\in Y_0(a)} 
B'_z(H'_{az,\widetilde{az}}\,v'_w,H'_{az,\widetilde{az}}\,v'_w)\,.
\end{equation*}

To this purpose, observe that, for  any element $w\in Y_0(a)$ 
there exists  a geodesic $[w,wt_b]'$ which starts
 at the vertex $w$ and ends up in the cone $C(ab)$ for some $b\neq\inv a$
(see Proposition~\ref{prop:3.10} and Figure~1).  This geodesic is "minimal"
in the sense that $C'(w\bar{t_b})$
would fail to be in the cone $C(ab)$.  The
endpoints $wt_b$ of these geodesics, for all possible $b$,
 are exactly the terminal points $T'_w$ of the tree
$\mathcal{X}'_w$. 
Hence, for each $w\in Y_0(a)$, by Lemma~\ref{formerestrizione}
applied to the shadow $\mu[w,v'_w]$ at the point $w$ and the tree $\mcX'_w$,
one has
\begin{equation*}
B'_w(v'_w,v'_w)=\sum_{b\neq\inv a}B'_{wt_b}(v'_{wt_b},v'_{wt_b})\;.
\end{equation*}

We need now to compare  the two quantities
$B'_{wt_b}(v'_{wt_b},v'_{wt_b})$ and 
$B'_{z}(H_{az,\widetilde{az}}v'_{w},H_{az,\widetilde{az}}v'_{w})$.

By Proposition~\ref{prop:3.10} we have  seen
that such terminal vertices can be written as $wt_b=az$ for some $z\in Y(b)$
and that $\widetilde{az}_a=w$.  By definition of $H_{ba}$ one has
\begin{equation*}
B'_{wt_b}(v'_{wt_b},v'_{wt_b})=
B'_{z}(H_{az,\widetilde{az}}v'_{w},H_{az,\widetilde{az}}v'_{w})
\end{equation*}
where  we have used again Lemma~\ref{lem:last letter}.
Summing over $w\in Y_0(a)$ (or, that is the same, over $az\in Y_0(a)$), we obtain the desired assertion.
\end{proof} 

Let now  $\pi$ be the left regular action of $\Ff_a$ on $\mcH^\infty(V_a,H_{ba})$
and let $\mcH(V_a, H_{ba}, B_a)$ be the completion of $\mcH^\infty(V_a,H_{ba})$  
with respect to the norm induced by the $(B_a)$.

\medskip

We define now the intertwining operator 
\begin{equation*}
U:\mcH^\infty(V'_\alpha,H'_{\beta\alpha},B'_\alpha)\to\mcH^\infty(V_a,H_{ba},B_a)\,.
\end{equation*} 
For every $f\in \mcH^\infty(V'_\alpha,H'_{\beta\alpha})$ 
and a reduced word $xa$ in the alphabet $A$ we set
\begin{equation*}
(U f)(xa):=\sum_{y\in Y(xa)}f(y)\,.
\end{equation*}

To see that $U$ intertwines $\pi'$ to $\pi$ fix any $y\in\mcV$
and assume that $|y|\leq|x|+1$. For any such $x$ and $y$ one has
\begin{align*}
     \pi(y)U f(xa)
&=U f(\inv{y}xa)
  =\sum_{z\in Y(\inv{y}xa)}f(z)
  =\sum_{z\in \inv{y} Y(xa)}f(z)\\
&=\sum_{u\in Y(xa)}f(\inv{y}u)
  =U\big(\pi'(y) f\big)(xa)
\end{align*}
since $Y(\inv{y}xa)=\inv{y}Y(xa)$ if $|y|\leq|x|+1$. It follows that
$U \pi'(y)f (xa)$ and $\pi(y)U f(xa)$ differ only for a finite
set of values of $x$, and hence are equal in 
$\mcH^\infty(V_a,H_{ba})$.

We conclude with the following
\begin{theorem}
$U$ is unitary.
\end{theorem}
\begin{proof}
Assume that $f\in \mcH^\infty(V'_\alpha,H'_{\beta\alpha})$ is
multiplicative
for $|y|'\geq N$. We may also assume that $f$ is zero if $|y'|\leq N-1$.
By  the discussion after Lemma~\ref{lem:last letter} there exists an integer $k$ such that 
$|y|\leq k$ whenever $|y|'\leq N$.
Define 
\begin{equation*}
S^0_k=\{ z\in \Ff \st C'(z)\nsubseteq C(x)\text{ for all } x \text{ with } |x|=k\}
\end{equation*}
and 
\begin{equation*}
\mcS'(k)=\{e\}\cup S^0_k\cup\bigcup_{\substack{x\in\mcT\\ |x|=k}}Y(x)\,.
\end{equation*}
Arguing as in the proof of Proposition~\ref{Z'}
one can show that
$\mcS'(k)$ is a finite complete subtree in $\mcT'$ whose terminal vertices
are the elements of $Y(x)$ for all $x$ with $|x|=k$. Since every $y$ belongs to
$C'(y)$, we see that $\mcS'(k)$ contains the ball of radius $N$ about the
origin in $\mcT'$. Use now Lemma~\ref{lem:normterm} to conclude the proof.
\end{proof}

We conclude this section with an example illustrating the effect of a nontrivial change of generators 
on a given multiplicative representation.

\begin{example}\label{ex:spherical}  Let $\Gamma=\Ff_A$, where $A=\{a,b,a^{-1}, b^{-1}\}$.
Consider the 
change of generators 
 given by $\alpha=a$ and $\beta=ab$ and let $\pi_s$ be the
spherical series representation of Fig\`a--Talamanca and Picardello
\cite{FT-P}
 constructed from the set of generators $A'=\{\alpha,\inv{\alpha},\beta
\inv{\beta}\}$. Denote by $a'$, $b'$ the generic elements
of $A'$. In \cite{K-S3} it is shown that $\pi_s$ can be realized as a
multiplicative representation with respect to the following   matrix system:
\begin{align*}
 V_{a'}&=\CC &\forall a'\in A'\\
 H_{b'a'}&= 3^{-\frac12+is} &\forall a',b'\in A'\\
 B_{a'}(v,v)&=\frac{|v|^2}{4}=:\lambda &\period
\end{align*}
Moreover, in \cite{Pe-S} it is also shown that it is impossible to realize
$\pi_s$ as any spherical representation arising from the generators $a$ and
$b$.
We show here that it is however possible to
 realize $\pi_s$ as a multiplicative representation with respect
to the other generators $a$ and $b$.
In fact one can verify that
\begin{align*}
&Y(a)=\{\alpha,\beta\}\\
&Y(b)=\{\inv\alpha \beta\}\\
&Y({\inv a})=\{\alpha^{-2},\inv\alpha\inv\beta\}\\
&Y(\inv b)=\{\inv\beta\}
\end{align*}
According to Definition~\ref{serveperexss} the 
spaces $V_a$ and $V_{\inv a}$ are two dimensional while $V_b=V_{\inv b}=\CC$.
The matrices appearing in~\ref{matricicambio} are:
\begin{align*}
&H_{aa} =H_{\inv a\inv a}=\left(\begin{array}{cc}
\lambda & 0\\ 
\lambda & 0\\
\end{array}\right)\\
& H_{b\inv a}=H_{\inv b a}=\left(\begin{array}{cc}
\lambda & 0
\end{array}\right)\\
&H_{ba}= H_{\inv b\inv a}=
\left(\begin{array}{cc}
0 & 1
\end{array}\right)\\
&H_{bb}=H_{\inv b\inv b}=\lambda^2&&\\
&H_{ab}=H_{\inv a\inv b}=\left(\begin{array}{c}
\lambda\\
\lambda
\end{array}\right)\\
&H_{a\inv b} = H_{\inv a b}=\left(\begin{array}{c}
\lambda^2\\
\lambda^2
\end{array}\right)\\ \period
\end{align*}
Let $W_a$ (respectively $W_{\inv a}$) denote the subspace of $V_a$ 
(respectively $V_{\inv a}$) generated by the vector $(1,1)$.
The reader can verify that the subspaces 
$W_a$, $W_{\inv a}$, $W_b=V_b=\CC$ and $W_{\inv b}=V_{\inv b}=\CC$ 
constitute an invariant subsystem and that the quotient system has Perron--Frobenius eigenvalue zero. 
According to Lemma~\ref{rho<1} the representation $\pi_s$
is equivalent to the multiplicative representation constructed from
 the subsystem $W$.
\end{example}

\section{Stability Under Restriction and Unitary Induction}\label{sec:restriction_induction}
In this section the set $A$ of generators for $\Gamma$ is fixed once and for all. 
As before, we write $\bar x$ for the (reduced) word obtained from $x$
by deleting the last letter of the reduced expression for $x$.
 Set also $\bar{a}=e$ if $a$ belongs to $A$.

\begin{definition}\label{Sch+D}
A Schreier system $S$ in $\Gamma$ is a nonempty subset of $\Gamma$ satisfying the 
following conditions:
\begin{enumerate}
\item $e\in S$;
\item if $x\in S$, then $\bar x\in S$.
\end{enumerate}

Assume that $\Gamma'$ is a subgroup of finite index in $\Gamma$. 
Essential in the following will be a choice of an appropriate fundamental domain
$D$
for the action of  $\Gamma'$ on the Cayley graph of $\Gamma$
with respect to a symmetric set of generators $A$. 
It is well known (see for example \cite[Chapter~VI]{Massey}) that one can
 choose in $\Gamma$ a set $S'$
 of representatives for the left cosets $\Gamma'\gamma$
 which is also a Schreier set. Identifying $S'$ with an appropriate
set of vertices $D$ of $\mcT$, it turns out that $D$ has the following 
properties:

\begin{itemize}
\item $D$ is a finite subtree containing $e$.
\item $D$ is a fundamental domain with respect 
to the left action on the vertices
of $\mcT$ in the sense that the set of vertices 
of $\mcT$ is the disjoint union of the
subtrees $x' D$ with $x'\in \Gamma'$. 
\end{itemize}
We shall refer to every such $D$ as to a {\it fundamental subtree}.
\end{definition}

Corresponding to that choice of $D$ one has also a natural choice of generators for $\Gamma'$, 
namely one can prove that $\Gamma'$ is generated by the set
\begin{equation}\label{gener}
A':=\big\{a_j'\in\Gamma:\;d(D,a_j'D)=1\big\}\,.
\end{equation}

We shall assume in this Section that $D$ is a fixed fundamental 
subtree and  that $A'$ is the corresponding generating set defined
as in \ref{gener}. We write $a',b',\dots$ to denote a generic
element of $A'$.


The following lemma  summarizes the 
properties of the translates of $D$  which will be used in several 
occasions to build
finite complete subtrees.

\begin{lemma}\label{prelim-tree}  Let $\gamma'a'\neq e$ be a reduced word in $\Gamma'$.
\begin{enumerate}
\item There exists $x\in\Gamma$ such that $\gamma'a'D\subset C(x)$ but $\gamma'D\not\subset C(x)$.
Moreover $\gamma'a'b'D\subset C(x)$ for all $b'$ such that $a'b'\neq e$.
\item The geodesic in $\mcT$ connecting $\gamma'a'D$ and $e$
crosses $\gamma'D$.
\end{enumerate}
\end{lemma}

\begin{proof}
Let $a'\in A'$ be a generator of $\Gamma'$ and $D$ a fundamental subtree. 
Let $x(a')\in a'D$ be the vertex of $a'D$ closest to $D$.
Since the distance between $D$ and $a'D$ is one, there exists a unique
edge $(x,x(a'))$ such that $x\in D$ and $x(a')\in a'D$. 
We claim that
$a'D\subset C(x(a'))$. Assume the contrary: namely assume
that  there exists $v\in a'D$
whose reduced word does not start with $x(a')$. Since $a'D$ is a subtree it must
contain the geodesic $[v,x(a')]$ connecting $v$ to $x(a')$, but this is 
impossible since $x\in [v,x(a')]$. 
Let $b'\in A'$ be such that $a'b'\neq e$. Denote by $(w,w')$ 
($w\in a'D$, $w'\in a'b'D$)
 the unique edge connecting 
$a'b'D$ to $a'D$. If $a'b'D\not\subset C(x(a'))$ it must be $w=x(a')$
and $w'=x$, which is impossible. By induction one has $a'\gamma'D
\subset C(x(a'))$ for every $\gamma'$ so that $a'\gamma'=1+|\gamma'|$. 

Let now $\gamma'a'$ be a reduced word in $\Gamma'$ and let
$x(\gamma'a')$ denote the vertex of $\gamma'a'D$ closest to $D$.
Translating the picture by $\inv{\gamma'}$ one can see that 
$\inv{\gamma'} x(\gamma'a')=x(a')$, that is
\begin{equation}\label{vertice+vicino}
x(\gamma'a')=\gamma'x(a')\period
\end{equation}
Since we have
\begin{equation*}
\gamma'a'D\subset \gamma'C(x(a')
\end{equation*}
(1) will be proved as soon as we show that  $\gamma'C(x(a')=C(\gamma'x(a'))$.
Let $\inv{d'}$ denote the last letter of $\gamma'$, so that $\inv{d'}\neq \inv{a'}$. 
Since  the two subtrees $d'D$ and $a'D$ are both at distance
one from $D$ they cannot be contained in the same cone:
so that neither $x(a')$ is the first part of $x(d')$ nor the converse. 
In particular $x(a')$ does not belong to the geodesic, in $\mcT$,
$[e,\inv{\gamma'}]$ so that, by Lemma~\ref{lem:cones},  $\gamma'C(x(a')=C(\gamma'x(a'))$.

To complete the proof observe that, since $a'D\subset C(x(a'))$ and $e\in D$,
 the geodesic connecting $D$ and $a'b'D$ must cross $x(a')$

\end{proof}

\subsection{Stability Under Restriction}\label{subsec:restriction}

The goal of this section is to prove the following:

\begin{theorem}\label{res}
Assume that $\Gamma$ is a finitely generated free group and $\Gamma'\subseteq \Gamma$
is  a subgroup of finite index.
If $\pi\in\Mult{ }$, then the restriction
of $\pi$ to $\Gamma'$ belongs to $\mathbf{Mult}(\Gamma')$.
\end{theorem}

\medskip

Choose $D$ and $A'$ as in Definition~\ref{Sch+D}.
Although $D$ is a finite subtree, it is not complete.  
The strategy of the proof 
consists of completing $D$ to a complete subtree $D'$, then translating $D'$
by a generator of $\Gamma'$, so that most of it
(in fact, all of it with the exception of the unique edge closer to the identity)
is contained in a cone at distance one from $D$.  A wise definition of
$(V_{a'},H_{b'a'})$ and $B_{a'}$, together with the help of a shadow
supported on the cone, will provide the construction of a matrix system with 
inner product for the subgroup $\Gamma'$.

Let, as in the proof of Lemma~\ref{prelim-tree}, denote by
$x(a')$ the vertex of $a'D$ closest to $D$.
Let $D'$ be the  subtree obtained by adding to $D$ 
the vertices $x(a')$ (and the relative edges) corresponding to all $a'\in A'$.
Write  $x(a')$ in the generators  of  $\Gamma$ 
and denote by $q(a')$ the last letter of its reduced expressions,
that, with the notation used in \eqref{normfv}, we have that $q(a')=t(x(a'))$.

\begin{lemma}\label{chiave} Let $D$, $D'$, $x(a')$ as above.
\begin{enumerate}
\item 
The subtree $D'$ is complete and its terminal vertices consist
of exactly all the $x(a')_{a'\in A'}$.
\item For every  $a',b'\in A'$,  the vertex
of $a'b'D$ closest to $a'D$ is $a'x(b')$. 
\item Assume that $a'b'\neq e$.  Then the geodesic joining
$e$ and $a'x(b')$ crosses $x(a')$ and the last letter of $a'x(b')$ is
$q(b')$.
\end{enumerate}
\end{lemma}

\begin{proof}
(1) Let $v\in D$ and assume that $v'$ is a neighbor of $v$. If $v'\notin D$ there
exists $x'\in\Gamma'$ and $u\in D$ such that $v'=x' u$. Hence
the distance between $D$ and $x'D$ is one: this implies that
$x'=a'$ for some $a'\in A'$ and $v'=x(a')$. This proves that
 every vertex of $D$ is an interior vertex of $D'$. Choose now any
$x(a')$ and consider its $q+1$ neighbors: one of them belongs to $D$ and
the  others, being at distance two from $D$, cannot be in $D'$. This proves 
that $D'$ is complete with  terminal vertices $x(a')_{a'\in A'}$.

\smallskip
\noindent
(2) follows immediately from \eqref{vertice+vicino}.
In particular the vertex of $a'b'D$ closest
to $a'D$ is $a'x(b')=x(a'b')$.

\smallskip
\noindent
(3) By Lemma~\ref{prelim-tree}, the geodesic joining
$e$ and $x(a'b')$, crosses $x(a')$.
In term of the generators of $\Gamma$ this means that $x(a')$ is the first
piece of the reduced word for $a'x(b')$ and, in particular,
passing from $x(a')$ to $a'x(b')$, the last letter 
of $x(a')$ is not canceled.
To prove the second assertion, observe that $e$ does not belong to $\inv{x(b')}\inv{(a')}D$.
In fact, if it did, one would have $e=\inv{x(b')}\inv{(a')}\xi_0$ for some
$\xi_0\in D$: but since we also have $x(b')=b'\xi_1$ this would imply that $\xi_0=\xi_1$
and $b'=\inv{(a')}$.
Hence the subtree $\inv{x(b')}\inv{(a')}D$ is contained in the cone $C(c)$ 
for some $c\in A$. Since 
\begin{equation*}
d(\inv{x(b')}D,
\inv{x(b')}\inv{(a')}D)=d(D,\inv{(a')}D)=1\,,
\end{equation*}
the subtree $\inv{x(b')}D$ is at distance one from  $\inv{x(b')}\inv{(a')}D$.
This is possible only in two ways: either $\inv{x(b')}D$ is contained in $C(c)$ 
or $\inv{x(b')}D$ contains the identity $e$. 
The second possibility is ruled out because $x(b')\notin D$. This implies that
the last letter of $x(b')$ is the same as the last letter of $a'x(b')$.
\end{proof}

We collect here the results as they will be needed later.

\begin{corollary}\label{sbgp}
With the above notation the subtree $a'D'$ is a non-elementary tree based at $x(a')$ 
whose terminal vertices are $T(a'D')=\{a'x(b'):\,b'\in A'\}$.
The terminal vertex closest to $e$ is $a'x(a'^{-1})$, so that
\begin{equation*}
T_e(a'D')=\{a'x(b'):\,b'\in A',\,a'b'\neq e\}
\end{equation*} 
and
\begin{equation}\label{geod}
a'x(b')=x(a')a_1a_2\dots a_kt(b')=a'x(a'^{-1})q(a')a_1a_2\dots a_kq(b')
\end{equation}
is the reduced expression of $a'x(b')$ in the alphabet $A$.
\end{corollary}

We are now ready to define the matrix system $(V_{a'},H_{b'a'})$.

\begin{definition}
With \eqref{geod} in mind, we set
\begin{equation*}
\begin{aligned}
&V_{a'}:=V_{q(a')}\,,\quad \text{and}\\
&H_{b'a'}:=\begin{cases}
                                  H_{q(b')a_k}\dots H_{a_2a_1}H_{a_1q(a')}&\text{ if }b'a'\neq e\\
                                  0&\text{ if }b'a'=e\,.
                   \end{cases}
\end{aligned}
\end{equation*}
\end{definition}

\begin{lemma}
The tuple $(B_{a'})_{a'\in A'}$ defined by 
\begin{equation*}
B_{a'}:=B_{q(a')}
\end{equation*}
is compatible
with the matrix system $(V_{a'},H_{b'a'})$.
\end{lemma}

\begin{proof}
We need to prove that, for every $v_{a'}\in V_{a'}$ 
\begin{equation}\label{tesi}
B_{a'}(v_{a'},v_{a'})=\sum_{b'\st a'b'\neq e}B_{b'}
(H_{b'a'}v_{a'},H_{b'a'}v_{a'})\,.
\end{equation}
Let $\mu[x(a'),v_{a'}]$ 
be  the shadow as in Definition~\ref{shadow}.
Since by definition 
\begin{equation*}
B_{a'}(v_{a'},v_{a'})=\big\|\mu[x(a'),v_{a'}](x(a'))\big\|^2\,,
\end{equation*}
showing \eqref{tesi} is equivalent to showing that 
\begin{equation*}
\big\|\mu[x(a'),v_{a'}](x(a'))\big\|^2=
\sum_{b':\,a'b'\neq e}\big\|H_{b'a'}\mu[x(a'),v_{a'}](x(a'))\big\|^2\,.
\end{equation*}
Moreover, since $\mu[x(a'),v_{a'}]$  is multiplicative, 
according to the definition of $H_{b'a'}$ we have
\begin{equation}\label{shadowmult}
\mu[x(a'),v_{a'}](a'x(b'))=H_{b'a'}\mu[x(a'),v_{a'}](x(a'))\,.
\end{equation}

By Lemma~\ref{formerestrizione}, Corollary~\ref{sbgp} and \eqref{shadowmult} it follows that
\begin{equation*}
\begin{aligned}
  \big\|\mu[x(a'),v_{a'}](x(a'))\big\|^2
&=\sum_{t\in T_e(a'D')}\big\|\mu[x(a'),v_{a'}](t)\big\|^2\\
&=\sum_{b':\,b'a'\neq e}\big\|\mu[x(a'),v_{a'}](a'x(b'))\big\|^2\\
&=\sum_{b':\,a'b'\neq e}\big\|H_{b'a'}\mu[x(a'),v_{a'}](x(a'))\big\|^2\,,
\end{aligned}
\end{equation*} 
which completes the proof.
\end{proof}

We need to define now the intertwining operator between the restriction $\pi|_{\Gamma'}$ to $\Gamma'$ of the
representation $\pi$ on $\mcH(V_a,H_{ba},B_a)$ and the representation $\rho$ of $\Gamma'$ on
$\mcH(V_{a'},H_{b'a'},B_{a'})$ defined by 
\begin{equation*}
\rho(x')f(y'):=f(\inv {x'}y')\,,
\end{equation*}
for $x',y'\in\Gamma'$ and $f\in\mcH(V_{a'},H_{b'a'},B_{a'})$.

\begin{definition}
Let $f\in\mcH^\infty(V_a,H_{ba},B_a)$. If $x'=y'a'\in\Gamma'$ with $a'\in A'$ 
and  $|x'|_{\Gamma'}=|y'|_{\Gamma'}+1$ 
(in the word metric with respect to the generators $A'$), define
\begin{equation*}
(Uf)(x'):=f\big(y'x(a')\big)\period
\end{equation*}
\end{definition}

\begin{proof}[Proof of Theorem~\ref{res}] It is easy to check that the operator
$U$ maps the restriction to $\Gamma'$ of multiplicative functions in $\mcH^\infty(V_a,H_{ba},B_a)$
to multiplicative functions in $\mcH^\infty(V_{a'},H_{b'a'},B_{a'})$.
In fact, if $x'=y'a'\in\Gamma$ with $a'\in\Gamma'$ and $|x'|_{\Gamma'}=|y'|_{\Gamma'}+1$, then
\begin{equation*}
(Uf)(x')=f\big(y'x(a')\big)\in V_{t(x(a'))}=V_{q(a')}\,.
\end{equation*}
Moreover, if $y'a'b'\in\Gamma'$ with $a',b'\in A'$ and $|y'a'b'|_{\Gamma'}=|y'|_{\Gamma'}+2$, then
\begin{equation*}
(Uf)(y'a'b')=f\big(y'a'x(b')\big)=H_{b'a'}\big(f(y'a')\big)\,.
\end{equation*}
Furthermore, it is straightforward to check that 
\begin{equation*}
U\big(\pi|_{\Gamma'}(x')f\big)=\rho(x')(Uf)\,,
\end{equation*}
thus completing the proof.
\end{proof}

\subsection{Stability Under Unitary Induction}\label{subsec:induction}
The goal of this section is to prove the following

\begin{theorem}\label{ind}
Assume that $\Gamma$ is a finitely generated free group and $\Gamma'\leq\Gamma$
is  a subgroup of finite index.
If $\pi'\in\mathbf{Mult}(\Gamma')$ then $\ind{\pi'}$ is in the class $\Mult{ }$.
\end{theorem}

Let $\Gamma'\leq\Gamma$ be a subgroup of finite index and let $D$ be a fundamental
subtree for the action of $\Gamma'$ on $\mcT$.
By Theorem~\ref{generators} we may assume that 
$A'$ is the generating set of $\Gamma'$
corresponding to $D$ as in \eqref{gener}.

Suppose that we are given a matrix system with inner
products $(V_{a'}, H_{b'a'},B_{a'})$ relative to $\Gamma'$ and hence
a representation $\pi'$ of the class  $\mathbf{Mult}(\Gamma')$
acting on $\mcH_s:=\mcH(V_{a'},H_{b'a'},B_{a'})$.  
Because of Theorem~\ref{decompo} we may always assume that
the system is irreducible.
Let $\ind{\pi'}$ denote the induced representation acting on $\ind{\mcH_s}$.
We recall that  
\begin{equation*}
\ind{\mcH_s}
:=\big\{f:\Gamma\to\mcH_s:\,\pi'(h)f(g)=f(g\inv{h}),\text{ for all }h\in\Gamma',g\in\Gamma\big\}\,,
\end{equation*}
on which $\Gamma$ acts by
\begin{equation*}
\big(\ind{\pi'}(g_0)f\big)(g):=f(\inv{g_0}g)\,,
\end{equation*}
for all $g_0,g\in\Gamma$.  
In particular $f(g)$ is uniquely determined by its values 
on a set of representatives for the right cosets of $\Gamma'$ in $\Gamma$,
which, with our choice of generators of $\Gamma'$, 
can  also be taken to be the fundamental subtree $D$.

Denote by $\mcH_s^\infty:=\mcH^\infty(V_{a'},H_{b'a'},B_{a'})$ the dense subspace $\mcH_s$
consisting of multiplicative functions and define,  with a slight abuse of notation, the dense subset
\begin{equation*}
\begin{aligned}
\ind{\mcH_s^\infty}:=\big\{f:\Gamma\to\mcH^\infty(V_{a'},H_{b'a'},B_{a'}):\,\pi'(h)f(g)=f(g\inv{h}),&
\\\text{ for all }h\in\Gamma',g\in\Gamma\big\}&
\end{aligned}
\end{equation*}
which, by definition of $\mcH_s^\infty$, can be identified with
\begin{equation*}
\begin{aligned}
\ind{\mcH_s^\infty}&\cong\big\{\varphi:D\cdot\Gamma'\to\coprod_{a'\in A'} 
V_{a'}:\,\pi'(h)\varphi(g)=\varphi(g\inv{h}),\\\text{ for all }&h\in\Gamma',g\in\Gamma\text{ and }
        \varphi\text{ is multiplicative as a function of }\Gamma'\big\}
\end{aligned}
\end{equation*}
via the map $f\mapsto\Phi(f)$, where
$\Phi(f)(x):=f(u)(h)$, for $x=uh$,  with $h\in\Gamma'$ and $u\in D$. 
The invariance property of functions in $\ind{\mcH_s^\infty}$ imply that $\Phi(f)$ is well defined.

We want to show that there exists a matrix system with inner product $(V_a,H_{ba},B_a)$ on $\Gamma$ so that
$\ind{\pi'}$ is equivalent to a multiplicative representation $\pi$ on $\mcH(V_a,H_{ba},B_a)$.  
The vector spaces $V_a$ will be direct sums of possibly multiple copies of the vector spaces $V_{a'}$,
according to some appropriately chosen "coordinates" on subsets of the cones $C(a)$. 
To this purpose, let us define for any   generator $a$ of $\Gamma$, the set
\begin{equation*}
P(a)=(\inv{D}\cdot A')\cap C(a)\,,
\end{equation*}
where $\inv{D}=\{\inv u\st u\in D\}$.


The following lemma is technical, but only specifies the multiplicative property of the chosen coordinates.

\begin{lemma}\label{3.16} Let us fix $a\in A$ and $v\in D$.
\begin{enumerate}
\item Assume that $va^{-1}\in D$ and let $c'\in A'$ be any generator.
Then $av^{-1}c'\in P(a)$ if and only if $v^{-1}c'\in P(b)$ for some $b\in A$ with $ab\neq e$.
\item Assume that $va^{-1}\notin D$.  Then
\begin{enumerate}
\item there exists $c'\in A'$ and $u\in D$ such that $av^{-1}=u^{-1}c'\in P(a)$;
\item furthermore for every $d'\in A'$ such that $c'd'\neq e$, there exists a unique $b\in A$ with $ab\neq e$ such that 
        $v^{-1}d'\in P(b)$.
\end{enumerate}
\end{enumerate}
\end{lemma}

\begin{proof}  (1) Let $b\in A$ be such that $v^{-1}c'\in P(b)$.  Then in particular $v^{-1}c'$ starts with $b$
and hence $av^{-1}c'\in C(a)$ if $ab\neq e$.  Since by hypothesis $va^{-1}\in D$, it follows that 
$av^{-1}c'\in P(a)$.  

Conversely, let $b\in A$ be such that $v^{-1}c'\in C(b)$.  Since $av^{-1}c'\in P(a)$, it follows that $ab\neq e$.
Moreover, since $v\in D$, we have that $v^{-1}c'\in P(b)$.

\medskip
\noindent
(2a) Since $v\in D$ but $va^{-1}\notin D$ and $D$ is  a Schreier system, then $|va^{-1}|=|v|+1$,
that is $d(va^{-1},D)=1$.  By \eqref{gener}, there exist $u\in D$ and $(c')^{-1}\in A'$ such that $va^{-1}=(c')^{-1}u$,
from which it follows that $av^{-1}=u^{-1}c'\in P(a)$. 

\medskip
\noindent
(2b) Choose $d'\in A'$.  By \eqref{gener}, $D$ and $d'D$ are disjoint subtrees at distance one from each other.
We claim that if $d'\neq (c')^{-1}$, neither of their translates $av^{-1}D$ and $av^{-1}d'D$ contains the identity $e$.
In fact, if $e$ were to belong to $av^{-1}D$, we would have that $va^{-1}\in D$, which is excluded by hypothesis.
If on the other hand $e$ were to belong to $av^{-1}d'D$, then we would have that for some $u_0\in D$,
$av^{-1}=u_0^{-1}(d')^{-1}$.  But by (2a) we know that $av^{-1}=u^{-1}c'$, so that,
by uniqueness of the decomposition, one would conclude that $c'=(d')^{-1}$, which is also excluded by hypothesis.

Hence both subtrees are contained in some cone $C(b)$, where $b\in A$ and, since they are at distance one from each other,
this cone must be the same for both.  But since $v\in D$, then $a\in av^{-1}D$, so that $av^{-1}D$, and hence $av^{-1}d'D$,
are contained in $C(a)$.  

Since $e\in D$, this means in particular that $av^{-1}d'\in C(a)$, so that $v^{-1}d'\in C(b)$ for some $b$
such that $ab\neq e$.  Hence $v^{-1}d'\in P(b)$.
\end{proof}

We are now ready to define the matrix system $(V_a,H_{ba})$.

\begin{definition}\label{va}
For every $u\in D$ and $a$ in $A$ let  $V_{u,a}$ 
be the direct sum of the spaces $V_{c'}$  for
all $c'$ such that $\inv{u}c'$ belongs to $P(a)$, namely
\begin{equation*}
V_{u,a}:=\bigoplus \big\{V_{c'}\st\,\, c'\in A'\text{ and } \inv{u}c'\in P(a)\big\}\,,
\end{equation*}
and set
\begin{equation}\label{eq:va-induced}
V_a:=\bigoplus_{u\in D} V_{u,a}=\bigoplus\big\{V_{c'}:\,\,u\in D,\,c'\in A'\text{ and }u^{-1}c'\in P(a)\big\}\,.
\end{equation}
\end{definition}

In other words, we can think of the $V_a$'s as consisting of blocks, corresponding to elements
$u\in D$ each of them containing a copy of $V_{c'}$ whenever $u^{-1}c'\in P(a)$.  
With this definition of the $V_a$'s, we can now define a map

\begin{equation*}
U:\ind{\mcH_s^\infty(V_{a'},H_{b'a'})}
\to\big\{\Gamma\to\bigoplus_{a\in A}V_a\big\}
\end{equation*}
with the idea in mind that the target will have to be the space of multiplicative functions on some
matrix system with inner product $(V_a,H_{ba},B_a)$.  Fix $a\in A$ and let $u^{-1}c'\in P(a)$.
Then for all $x\in\Gamma$ such that $|xa|=|x|+1$ and for 
$f\in\ind{\mcH_s^\infty(V_{a'},H_{b'a'})}$, we define $Uf(xa)$ 
to be the vector whose $(u,c')$-component
is given by

\begin{equation*}
Uf(xa)_{{u},c'}:=\Phi(f)(x\inv{u}c')
\end{equation*}

or, equivalently,

\begin{equation}\label{eq:J}
Uf(xa)=\bigoplus_{(u,c')}f(x\inv{u})(c')
\end{equation}

It is not difficult to convince oneself on how to construct the linear 
maps $H_{ba}$
so that the functions $Uf$ will be multiplicative: we give here an explanation,
and one can find the formula in \eqref{matrici}.

Since the functions $Uf$ will have to be multiplicative, if $|xab|=|x|+2$
they will have to satisfy

\begin{equation*}
f(xav^{-1}d')=(Uf)(xab)_{v,d'}=\big(H_{ba}(Uf)(xa)\big)_{v,d'}
\end{equation*}
whenever $v^{-1}d'\in P(b)$ for some $H_{ba}:V_a\to V_b$ to be specified.
Thinking of the "block decomposition" alluded to above, the linear maps $H_{ba}$
will also be block matrices that will perform three kinds of operations
on a vector $w_a\in V_a$ with coordinates $w_a=(w_{u,c})_{u^{-1}c'\in P(a)}$.
\begin{enumerate}
\item[--] If there exists $d'\in A'$ such that for some $v\in D$, $a^{-1}vd'\in P(a)$ and
$v^{-1}d'\in P(b)$, (see Lemma~\ref{3.16} (1)), then $H_{ba}$ will just move the 
$(va^{-1},d')$-component of $w_a$ to the $(v',d')$-component of $H_{ba}w_a$.
According to Lemma~\ref{3.16}(1) this happens precisely when $va^{-1}\in D$.
\item[--] If on the other hand for $u,v\in D$, $u^{-1}c'\in P(a)$ and $v^{-1}d'\in P(b)$,
then $c'd'\neq e$ (cf. Lemma~\ref{3.16}(2)) and $H_{ba}|_{V_{u,c'}}:V_{u,c'}\to V_{v,d'}$
will be nothing but $H_{d'c'}$. 
\item[--] In all other cases $H_{ba}$ will be set equal to zero .
\end{enumerate}

More precisely we define
\begin{equation}\label{matrici}
(H_{ba}w_a)_{{v},d'}:=
\begin{cases}
\begin{aligned}
&(w_a)_{v\inv{a},d'}&\qquad\text{if }v\inv{a}\in D\\
&H_{d'c'}(w_a)_{u,c'} &\hphantom{XX}\text{if } v\inv{a}\notin D \text{ and }a^{-1}v=u^{-1}c'\\
&0&\text{otherwise}\,.
\end{aligned}
\end{cases}
\end{equation}

That this makes sense follows directly from Lemma~\ref{3.16} as we explained above.

The definition of a tuple of positive definite forms is now obvious, namely
the $(u,c')$-component of $B_{a}$ is given by the following
\begin{equation}\label{ba}
({B_a})_{{u},c'}:=B_{c'}\qquad\text{where $\inv{u}c'\in P(a)$}
\end{equation}

\begin{proposition}
The tuple $(B_a)_{a\in A}$ is compatible with the system 
$H_{ba}$ defined in \eqref{matrici}.
\end{proposition}

\begin{proof}
We must check that, for every $ w_a\in V{a}$ one has
\begin{equation*}
{B_a}(w_a,w_a)=\sum_{b\st ab\neq e}
{B_b}(H_{ba}w_a,H_{ba} w_a)\,.
\end{equation*}
Remembering that, by definition of $V_a$ and $B_a$
\begin{equation}\label{eq:B_a}
{B_a}(w_a, w_a)=
\sum_{u\in F}\;\;\sum_{\inv{u}c'\in P(a)} B_{c'}\big((w_a)_{u,c'},(w_a)_{u,c'}\big)\,,
\end{equation}
we must prove that
\begin{equation}\label{two sums}
\begin{aligned}
&\sum_{u\in F}\sum_{\inv{u}c'\in P(a)} B_{c'}\big((w_a)_{u,c'},(w_a)_{u,c'}\big)=\\
&\sum_{b\st ab\neq e}\;\;
\sum_{v\in F}\;\;\sum_{\inv{v}d'\in P(b)} B_{d'}\big((H_{ba} w_a)_{v,d'},(H_{ba} w_a)_{v,d'}\big)\,.
\end{aligned}
\end{equation}
Fix $a$ in $A$ and define
\begin{equation*}
 D_a=\{u\in D\st u= v\inv{a}\;\; \text{for some $v\in D$}\}\,,
\end{equation*}
so that 
\begin{equation*}
{D_a\cdot a}=\{v\in D\st v=ua\;\; \text{for some $u\in D_a$}\}
\end{equation*}
is in bijective correspondence with $D_a$.

Split the sums on each side of \eqref{two sums} to obtain

\begin{equation}\label{four sums}
\begin{aligned}
&\sum_{u\in D_a}\;\;
        \sum_{\inv{u}c'\in P(a)} B_{c'}\big((w_a)_{u,c'},(w_a)_{u,c'}\big)\\
&+\sum_{u\in D\setminus D_a}\;\;
                \sum_{\inv{u}c'\in P(a)} B_{c'}\big((w_a)_{u,c'},(w_a)_{u,c'}\big)\\
&=\sum_{v\in {D_a\cdot a}}\;\;\sum_{b\st ab\neq e}\;\;
        \sum_{\inv{v}d'\in P(b)} B_{d'}\big((H_{ba} w_a)_{v,d'},(H_{ba} w_a)_{v,d'}\big)\\
&+\sum_{v\in D\setminus {D_a\cdot a}}\;\;\sum_{b\st ab\neq e}\;\;
\sum_{\inv{v}d'\in P(b)} B_{d'}
\big((H_{ba} w_a)_{v,d'},(H_{ba} w_a)_{v,d'}\big)\,.
\end{aligned}
\end{equation}

We will show the equality

\begin{equation}\label{somme interne}
\begin{aligned}
&\sum_{\inv{u}c'\in P(a)} B_{c'}\big((w_a)_{u,c'},(w_a)_{u,c'}\big)\\
&=\sum_{b\st ab\neq e}\;\;
        \sum_{\inv{v}d'\in P(b)} B_{d'}\big((H_{ba} w_a)_{v,d'},(H_{ba} w_a)_{v,d'}\big)\\
\end{aligned}
\end{equation}
in the two cases
\begin{enumerate}
\item  $u\in D_a$ and  $v=ua\in D_a\cdot a$,
\item $u\notin D_a$ and  $v=ua\notin D_a\cdot a$.
\end{enumerate}

Then \eqref{four sums} will follow by summing \eqref{somme interne}
once over $D_a$ and once over $D\setminus D_a$ and adding the resulting equations.

\medskip
\noindent
(1) Let $u\in D_a$ and $v\in D_a\cdot a$.  Then for a fixed $c'\in A'$ with 
$\inv{u}c'\in P(a)$, Lemma~\ref{3.16}(1) implies that 
\begin{equation*}
\begin{aligned}
&\sum_{c'\st a\inv{v}c'\in P(a)} B_{c'}\big((w_a)_{v\inv{a},c'},(w_a)_{v\inv{a},c'}\big)\\
&=\sum_{b\st ab\neq e}\;\;\sum_{c'\st\inv{v}c'\in P(b)} B_{c'}\big((w_a)_{v\inv{a},c'},(w_a)_{v\inv{a},c'}\big)\,,
\end{aligned}
\end{equation*}
so that 
\begin{equation*}
\begin{aligned}
&\sum_{c'\st \inv{u}c'\in P(a)} B_{c'}\big((w_a)_{u,c'},(w_a)_{u,c'}\big)\\
&=\sum_{c'\st a\inv{v}c'\in P(a)} B_{c'}\big((w_a)_{v\inv{a},c'},(w_a)_{v\inv{a},c'}\big)\\
&=\sum_{b\st ab\neq e}\;\;\sum_{c'\st\inv{v}c'\in P(b)} B_{c'}\big((w_a)_{v\inv{a},c'},(w_a)_{v\inv{a},c'}\big)\\
&=\sum_{b\st ab\neq e}\;\;
        \sum_{c'\st \inv{v}c'\in P(b)} B_{c'}\big((H_{ba} w_a)_{v,c'},(H_{ba} w_a)_{v,c'}\big)\\
&=\sum_{b\st ab\neq e}\;\;
        \sum_{d'\st \inv{v}d'\in P(b)} B_{d'}\big((H_{ba} w_a)_{v,d'},(H_{ba} w_a)_{v,d'}\big)\,,
\end{aligned}
\end{equation*}
where the next to the last equation comes from the definition of the $H_a$ and the last 
just from renaming the variable.

\medskip
\noindent
(2) Fix now any $v$ in $D\setminus D_a\cdot a$ and write $a\inv{v}=\inv{u}c'$ (Lemma~\ref{3.16}(2a)).
Choose any $d'$ with $c'd'\neq e$ and let $b\in A$ with $ab\neq e$
be the unique $b$ such that $\inv{v}d'\in P(b)$ (Lemma~\ref{3.16}(2b))
By definition of $B_a$
\begin{equation*}
(H_{ba}w_a)_{v,d'}= H_{d'c'}(w_a)_{u,c'}\,.
\end{equation*}
To every $b$ corresponds a subset $A'_b$
 of $A'$ consisting of all $d'$ such that
$\inv{v}d'$ belongs to $P(b)$ and we observed before that
$\bigcup_b A'_b=A'\setminus \inv{(c')}$. 
Hence
\begin{equation*}
\begin{aligned}
&\sum_{b\st ab\neq e}\;\;\sum_{\inv{v}d'\in P(b)} B_{d'}\big((H_{ba} w_a)_{v,d'},(H_{ba} w_a)_{v,d'}\big)=\\
&\sum_{b\st ab\neq e}\;\;\sum_{d'\in A'_b} B_{d'}(H_{d'c'}(w_a)_{u,c'},H_{d'c'}(w_a)_{u,c'})=\\
&\sum_{d'\in A'\setminus(c')^{-1}}B_{d'}(H_{d'c'}(w_a)_{u,c'},H_{d'c'}(w_a)_{u,c'})=
B_{c'}((w_a)_{u,c'},(w_a)_{u,c'})\,,
\end{aligned}
\end{equation*}
where the last equality is nothing but the compatibility of the $(B_{a})$.
In particular  to every $v$ in $D\setminus {D_a\cdot a}$ corresponds a unique $u$ in 
$D\setminus D_a$
and a unique $c'\in A'$ such that $\inv{u}c'\in P(a)$ and 
\begin{equation*}
\begin{aligned}
\sum_{u\in D\setminus D_a}
\sum_{b\st ab\neq e}
\sum_{\inv{v}d'\in P(b)} B_{d'}
\big((H_{ba} w_a)_{v,d'},(H_{ba} w_a)_{v,d'}\big)\\
=\sum_{v\in D\setminus D_a\cdot a}B_{c'}\big((w_a)_{u,c'},(w_a)_{u,c'}\big)\,.
\end{aligned}
\end{equation*}
\end{proof}

The upshot of the above discussion is that we have shown that the map $U$ takes values 
in the space of multiplicative functions.  
We still need to show that $U$ is an unitary operator and hence it extends to a
unitary equivalence between $\ind{\mcH(V_{a'},H_{b'a'},B_{a'})}$ and $\mcH(V_a,H_{ba},B_a)$.
The following theorem will complete the proof.

\begin{theorem}\label{Junitary}
Let $V_a$, $H_{ba}$ and $B_a$ be as in \eqref{eq:va-induced}, \eqref{matrici} and
\eqref{ba} and let 
\begin{equation*}
U:\ind{\mcH^\infty(V_{a'},H_{b'a'},B_{a'})}\to \mcH^\infty(V_a,H_{ba},B_a)
\end{equation*}
be as in \eqref{eq:J}. Then
 $U$ is an unitary operator and hence it extends to a
unitary equivalence 
\begin{equation*}
U:\ind{\mcH(V_{a'},H_{b'a'},B_{a'})}\to\mcH(V_a,H_{ba},B_a)\,.
\end{equation*}
\end{theorem}
\begin{proof}
Let us simply write as before $\mcH^\infty_s$ for $\mcH^\infty(V_{a'},H_{b'a'}, B_{a'})$
and $\mcH^\infty$ for $\mcH^\infty(V_{a},H_{ba}, B_{a})$.

For every $f\in\ind{\mcH_s^\infty}$ we have by definition of the induced norm that
\begin{equation*}
\norm{f}^2_{\ind{\mcH_s^\infty}} =\sum_{u\in D}\norm{f(u)}_{\mcH_s^\infty}^2\,,
\end{equation*}
and, since the above sum is orthogonal,  we may assume that
$f$ is supported on $z\cdot\Gamma'$ for some $z\in D$. 

For such an $f$
it will be hence enough to show that 
\begin{equation*}
\norm{Uf}^2_{\mcH^\infty}=\norm{f(z)}^2_{\mcH_s^\infty}\,.
\end{equation*}

Using the definition of the norm in \eqref{eq:norm} as well as the definitons of  
$U$ in \eqref{eq:J} and of $B_a$ in \eqref{eq:B_a} we obtain that for $N$ large enough

\begin{equation*}
\begin{aligned}
     \norm{Uf}^2_{\mcH^\infty}
&=\sum_{a\in A}\sum_{\substack{|x|=N\\|xa|=|x|+1}} B_a\big(Uf(xa),Uf(xa)\big)\\
&=\sum_{a\in A}\sum_{\substack{|x|=N\\|xa|=|x|+1}} \sum_{\inv uc'\in P(a)} B_{c'}\big(f(x\inv u)(c'),f(x\inv u)(c')\big)\,.
\end{aligned}
\end{equation*}

Since $f(z)\in\mcH_s^\infty$, there exists $M>0$ such that $f(z)$ is multiplicative outside the ball $B'(e,M)$ in $\mcT'$
of radius $M$.  To complete the proof it will be hence enough to show the following

\begin{lemma}\label{lem:subtree}  There exists a finite complete subtree $\mcS'\subset\mcT'$ containing $B'(e,M)$
whose terminal elements are 
\begin{equation*}
T(\mcS')=\{\gamma'=\inv zxy\in\Gamma':\,|x|=N,\,|xa|=N+1,\, y\in P(a)\}
\end{equation*}
\end{lemma}

Observe that since, according to the above lemma, $\gamma'\in T(\mcS')$ has the form
$\gamma'=\inv zx\inv uc'$ with $u\in D$ and $c'\in A'$, 
the invariance property of $f$ translates into the equality
\begin{equation*}
f(z)(\gamma')=f(x\inv u)(c')\,.
\end{equation*}
From this in fact, using Lemma~\ref{lem:normterm} and
denoting $\overline{\gamma'}$ to be as before the reduced word obtained by deleting the last letter 
(in $\Gamma'$) of $\gamma'$,
we deduce that

\begin{equation*}
\begin{aligned}
    \norm{f(z)}^2_{\mcH_s^\infty}
&=\sum_{\substack{\gamma'\in T(\mcS')\\ \gamma'=\overline{\gamma'}c'}}B_{c'}\big(f(z)(\gamma'),f(z)(\gamma')\big)\\
&=\sum_{a\in A}\sum_{\substack{|x|=N\\|xa|=|x|+1}} \sum_{\inv uc'\in P(a)} B_{c'}\big(f(x\inv u)(c'),f(x\inv u)(c')\big)\,,
\end{aligned}
\end{equation*}

thus concluding the proof.
\end{proof}

We need now to show Lemma~\ref{lem:subtree}.
We start recording the following obvious fact, which follows immediately from the observation 
that left translates of $D$ are subtrees (hence convex) and that cones are disjoint and convex.

\begin{lemma}\label{lem:tree-subtree} Let $\Gamma'\leq\Gamma$ be a subgroup of a free group
with associated trees $\mcT'\subset\mcT$ and let $D$ a  fundamental subtree
in $\mcT$.  
Then for any $w\in \Gamma$ we can write
\begin{equation*}\label{eq:tree}
\mcT=w\,B(e,N+1)\sqcup\bigsqcup_{\substack{|x|=N\\ |xa|=N+1}}w\,C(xa)
\end{equation*}
and
\begin{equation*}\label{eq:subtree}
\begin{aligned}
\mcT'=\big\{\gamma'\in\Gamma':\,\gamma'D\,\cap\, w\,B(e,N+1)\neq\emptyset\big\}\sqcup&\\
\sqcup\bigsqcup_{\substack{|x|=N\\ |xa|=N+1}}\big\{\gamma'\in\Gamma':\,\gamma'D\subseteq w\, C(xa)\big\}&\,.
\end{aligned}
\end{equation*}
\end{lemma}

Clearly there are finitely many $\gamma'\in \Gamma'$ such that $\gamma'D\,\cap\,w\ B(e,N+1)\neq\emptyset$,
but infinitely many $\gamma'\in\Gamma'$ such that $\gamma'D\subseteq w\, C(xa)$ for some fixed $x$ and $a$.
The right finiteness condition is imposed in the following lemma.

\begin{lemma}\label{lem:done} Fix any $z\in \Gamma$ and choose
$N>|z|$  large enough so  that 
$\gamma' D\cap \inv zB(e,N+1)\neq\emptyset$ for all $|\gamma'|\leq M$.  
Define
\begin{equation*}
\begin{aligned}
S'_0 &:= \{\gamma'\in\Gamma'\st \gamma'D\cap\inv{z}B(e,N+1)\neq\emptyset\}\,,\\
S'_t &:=\{ \gamma'\in\Gamma'\st \gamma'D\subseteq\inv{z} C(xa)\text{ for some }
x,\,a\text{ with } |xa|=N+1\\
 &\hphantom{hhhhhhhhhhh} \text{ and } \overline{\gamma'}D\nsubseteq\inv{z}C(xa)\}\\
\mcS'&:=S'_0\sqcup S'_t\,.
\end{aligned}
\end{equation*}
Then $\mcS'$ is a finite complete subtree (containing $B'(e,M)$), whose terminal vertices
are $T(\mcS')=S'_t$ and can be characterized as follows
\begin{equation*}
T(\mcS')=\{\gamma'=\inv zxy\in\Gamma':\,|x|=N,\,|xa|=N+1,\, y\in P(a)\}\,.
\end{equation*}
\end{lemma}

\begin{proof}[Proof of Lemma~\ref{lem:subtree}]
We shall prove a sequence of simple claims.  Notice that since $|z|< N$, 
then for all $x\in \Gamma$ and $a\in A$ such that $|xa|=|x|+1$, $xa$ does not belong 
to the geodesic between $e$ and $z$ and hence, according to Lemma~\ref{lem:cones},
$\inv zC(xa)=C(\inv zxa)$.

\medskip
\noindent
{\em Claim 1.} If $\gamma'\in S'_0$, then $\overline{\gamma'}\in S'_0$ and hence the set $S'_0$ is a subtree.

\smallskip
\noindent
{\em Proof:} Let $v\in\gamma'D\cap\inv z\,B(e,N+1)$ be a vertex and let $x_0=v,x_1,\dots,x_r=e$ be a sequence
of vertices of the unique geodesic in $\mcT$ from $x_0=v$ to $x_r=e$.  By convexity of $\inv zB(e,N+1)$, 
$x_j\in\inv zB(e,N+1)$ for all $0\leq j\leq r$.  Since $\gamma'D$ is a subtree,
the set $\{i\st 0\leq i\leq r,\,x_i\in\gamma'D\}$ is an interval, say $[0,i_0]\cap\ZZ$. 
Let $\gamma''\in\Gamma'$ be (the unique element) such that $x_{i_0+1}\in\gamma''D$.
Then by construction $d(\gamma'D,\gamma''D)=1$ so that $\gamma''=\overline{\gamma'}$
and $\overline{\gamma'}\,D\cap\inv zB(e,N+1)\neq\emptyset$, thus showing that $\overline{\gamma'}\in S'_0$.

\begin{center}\hskip-2cm
{\psfrag{dist}{$d(\gamma'D,\overline{\gamma'}D)=1$}
 \psfrag{g'd}{$\gamma'D$}
 \psfrag{g'bd}{$\overline{\gamma'}D$}
 \psfrag{x}{$x$}
 \psfrag{xa}{$xa$}
\includegraphics[width=0.6\linewidth]{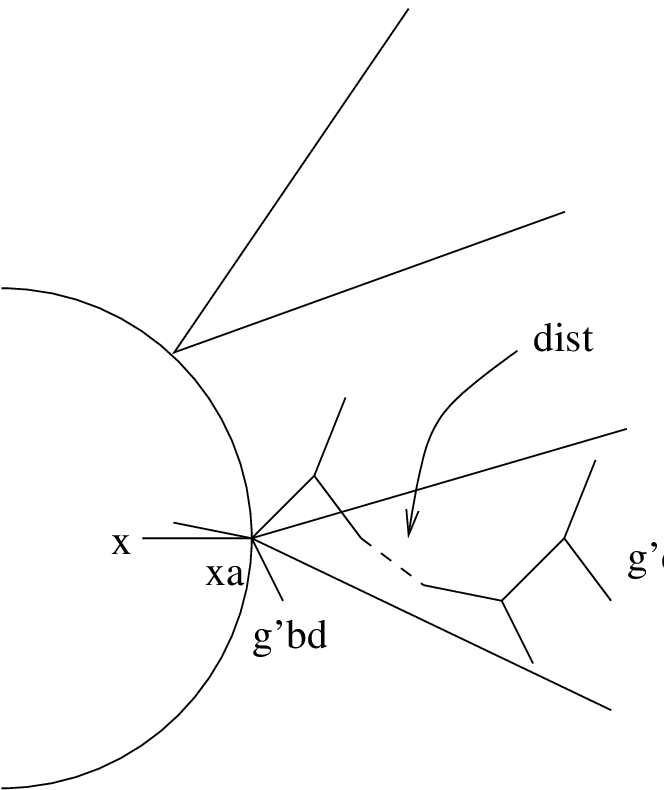}
}
\end{center}
\begin{center}{\sc Figure 2}: $\gamma'\in S'_t$ and $\overline{\gamma'}D\in S'_0$.\end{center}

\medskip
\noindent
{\em Claim 2.} If $\gamma'\in S'_t$, then $\overline{\gamma'}\in S'_0$ and hence the set $\mcS'$ is a subtree and 
$S'_t\subseteq T(\mcS')$.

\smallskip
\noindent
{\em Proof:}  Let $\gamma'\in S'_t$ and let $\gamma'D\subset\inv zC(xa)$ with $\overline{\gamma'}D\notin\inv zC(xa)$.
Lemma~\ref{lem:subtree} implies then immediately that $\overline{\gamma'}D\cap\inv z\,B(e,N+1)\neq\emptyset$ 
and hence $\overline{\gamma'}\in S'_0$.

\medskip
\noindent
{\em Claim 3.} The tree $\mcS'$ is complete and $S'_t= T(\mcS')$.

\smallskip
\noindent
{\em Proof:} Let $\gamma'\in S'_0$ and let $a'\in A"$ so that $|\gamma'a'|'=|\gamma'|'+1$.
If $\gamma'a'\notin S'_0$, then, by Lemma~\ref{lem:subtree}, $\gamma'a'D\in\inv zC(xa)$ 
for some $|x|=N$ and $|xa|=N+1$.  
On the other hand $\overline{\gamma'a'}D=\gamma'D\notin\inv zC(xa)$ and hence
$\gamma'\in S'_t$.

\medskip
\noindent
{\em Claim 4.} $T(\mcS')=\{\gamma'=\inv zxy\in\Gamma':\,|x|=N,\,|xa|=N+1,\, y\in P(a)\}$.

\smallskip
\noindent
{\em Proof:} By definition if $\gamma'\in S'_t$, then $\gamma'D\subseteq\inv z C(xa)$
and hence $\gamma'=\inv z xay$, for some $y\in\Gamma$.  However, since we have also that 
$\overline{\gamma'}D\nsubseteq\inv z C(xa)$, then $\inv zx\in\overline{\gamma'}D$.
Thus there exists $u\in D$ such that $\overline{\gamma'}=\inv zx\inv u$.  
The assertion now follows by completing $\gamma'$ with its last letter $c'\in A'$ 
in the reduced expression.
\end{proof}

\bibliographystyle{amsalpha}

\def\cprime{$'$}
\providecommand{\bysame}{\leavevmode\hbox to3em{\hrulefill}\thinspace}
\providecommand{\MR}{\relax\ifhmode\unskip\space\fi MR }
\providecommand{\MRhref}[2]{%
  \href{http://www.ams.org/mathscinet-getitem?mr=#1}{#2}
}
\providecommand{\href}[2]{#2}


\end{document}